%% file: PME_BlondelCancesSasadaSimon-corrections.tex
\documentclass[11pt,reqno, a4paper]{amsart}
\usepackage[utf8]{inputenc}
\usepackage{graphicx}
\usepackage{fancyhdr}
\usepackage{subfig}

\usepackage{amsthm}
\usepackage{mathabx}
\usepackage{mathrsfs}
\usepackage{amsmath,amssymb,color}
\usepackage{amsfonts, amscd, epsfig, amsmath, amssymb,enumerate}
\usepackage{tikz}
\usepackage{pgfplots}

\allowdisplaybreaks[1]

\usepackage[colorlinks=true,linkcolor=blue,citecolor=magenta]{hyperref}

\usepackage[normalem]{ulem}

\newcommand{\cm}[1]{#1}

\def\clem#1{{#1}}

\newcommand{\Clock}[5][shift={(0,0)}]%
{\begin{scope}[#1,scale={#2/5},transform shape]
\fill[color=cyan!30] (0,0) circle (5);
\foreach \a in {1,2,...,60}
   \draw[black] ($(6*\a:4.5)$) -- ($(6*\a:5)$);
\foreach \a in {5,10,...,60} \draw[black,fill,line width={#2*.4pt}]
      ($(6*\a:3.5)$) circle (2pt) -- ($(6*\a:5)$);
\foreach \a in {1,2,...,12}
   \node[text=black,fill=blue!50]
        at ($(90-\a*30:4.3)$) {\Large \bf \a};
\draw[->,>=stealth',line width={#2*1.3pt}]
   (0,0) -- ($(450-30*#3-0.5*#4-0.0083*#5:3.25)$);
\draw[line width={#2*1.1pt},gray]
   (0,0) -- ($(450-6*#4-0.1*#5:4.2)$);
\draw[->,>=stealth',line width={#2*.7pt}]
   (0,0) -- ($(450-6*#4-0.1*#5:4.45)$);
\draw[>->,>=stealth',line width={#2*.6pt},color=gray!80!blue]
   ($(270-6*#5:1)$) -- ($(450-6*#5:4.4)$);
\draw[black,line width={#2*.2pt},fill=gray] (0,0) circle (.2);
\draw[black,fill] (0,0) circle (.1);
\draw[black,line width={#2*.6pt}] (0,0) circle (5);
\end{scope}}

\newtheorem{theorem}{Theorem}[section]
\newtheorem{lemma}[theorem]{Lemma}
\newtheorem{proposition}[theorem]{Proposition}
\newtheorem{corollary}[theorem]{Corollary}
\newtheorem{remark}[theorem]{Remark}
\newtheorem{definition}[theorem]{Definition}
\newtheorem{assumption}[theorem]{Assumption}

\usepackage{tikz}
\usetikzlibrary{backgrounds}
\usetikzlibrary{patterns,fadings}
\usetikzlibrary{arrows,decorations.pathmorphing}
\usetikzlibrary{decorations}
\usetikzlibrary{calc}
\definecolor{light-gray}{gray}{0.95}
\usepackage{float}
\usepackage[colorlinks=true,linkcolor=blue,citecolor=magenta]{hyperref}

\def\centerarc[#1](#2)(#3:#4:#5){\draw[#1] ($(#2)+({#5*cos(#3)},{#5*sin(#3)})$) arc (#3:#4:#5);}

\numberwithin{equation}{section}

\def\be{\begin{equation}}
\def\ee{\end{equation}}
\def\bbT{{\mathbb T}}
\def\Cc{\mathcal{C}}
\def\p{\partial}
\def\d{{\rm d}}

\def\Pr{\varpi}

\newcommand{\mc}[1]{{\mathcal #1}}

\newcommand{\bb}[1]{{\mathbb #1}}

\newcommand{\eps}{\varepsilon}

\newcommand{\R}{\mathbb R}

\newcommand{\Z}{\mathbb Z}
\newcommand{\N}{\mathbb N}
\renewcommand{\P}{\mathbb P}
\newcommand{\T}{\mathbb T}
\newcommand{\E}{\mathbb E}
\renewcommand{\hat}{\widehat}


\title[Convergence of a degenerate microscopic dynamics to the PME]{Convergence of a degenerate microscopic dynamics to the porous medium equation}
\author{Oriane Blondel}
\address{Univ Lyon, CNRS, Universit\'e Claude Bernard Lyon 1, UMR5208, Institut Camille Jordan, F-69622 Villeurbanne, France}
\email{blondel@math.univ-lyon1.fr}
\author{Cl\'ement Canc\`es}
\address{Inria, Univ. Lille, CNRS, UMR 8524 - Laboratoire Paul Painlev\'e, F-59000 Lille}
\email{clement.cances@inria.fr}
\author{Makiko Sasada}
\address{Graduate School of Mathematical Sciences, University of Tokyo, 3-8-1, Komaba, Meguro-ku, Tokyo, 153--8914, Japan}
\email{sasada@ms.u-tokyo.ac.jp}
\author{Marielle Simon}
\address{Inria, Univ. Lille, CNRS, UMR 8524 - Laboratoire Paul Painlev\'e, F-59000 Lille}
\email{marielle.simon@inria.fr}

\begin{document}

\begin{abstract}
We derive the porous medium equation from an interacting particle system which belongs to the family of exclusion processes, with nearest neighbor exchanges. The particles follow a degenerate dynamics, in the sense that the jump rates can vanish for certain configurations, and there exist blocked configurations that cannot evolve. 
In \cite{GLT} it was proved that the macroscopic density profile in the hydrodynamic limit is governed by the porous medium equation (PME), for initial densities uniformly bounded away from $0$ and $1$.
In this paper we consider the more general case where the density can take those extreme values. 
In this context, the PME solutions display a richer behavior, like moving interfaces, finite speed of propagation and breaking of regularity.
As a consequence, the standard techniques that are commonly used to prove this hydrodynamic limits cannot be straightforwardly applied to our case. We present here a way to generalize the \emph{relative entropy method}, by involving approximations of solutions to the hydrodynamic equation, instead of exact solutions. 
\end{abstract}
\maketitle

{\color{red} 
\fbox{\begin{minipage}{0.9\textwidth}
{\Large \begin{center}\bf CAREFUL: this version is uncomplete.  \end{center}

There is one step missing, which makes the proof fail. We are currently trying to fill the gap. For any question or remark, please contact one of the authors.
}
\end{minipage}}}

\section{Introduction}

The derivation of \emph{macroscopic} partial differential equations from \emph{microscopic} interacting particle systems has aroused an intense research activity in the past few decades. 
In particular, the family of conservative interacting particle systems with exclusion constraints is rich enough to provide significant results. One aims at showing that the macroscopic density profile for these models  evolves under time rescaling according to some deterministic partial differential equation (PDE). The space-time scaling limit procedure which is at play here is called \emph{hydrodynamic limit}. The simplest example in that family is the \emph{symmetric simple exclusion process} (SSEP), for which the macroscopic hydrodynamic equation is the linear heat equation \cite[Chapter 2.2]{kl}. The purpose of this article is to present a new tool for the derivation of the hydrodynamic limit, when the macroscopic PDE belongs to the class of nonlinear diffusion equations which are not parabolic.

In \cite{GLT}, Gon\c{c}alves \emph{et al.} designed an exclusion process with \emph{local kinetic constraints}, in order to obtain the \emph{porous medium equation} (PME) as the macroscopic limit equation. The class of kinetically constrained lattice gases has been introduced in the physical literature in the 1980's (we refer to \cite{bt,rs} for a review) and is used to model liquid/glass transitions. The PME is a partial differential equation which reads in dimension one as 
\begin{equation}
\label{eq:pmeintro}
\partial_t\rho = \partial_{uu}(\rho^m),
\end{equation}
\cm{and here we assume that $m$ is a positive integer which satisfies $m\geqslant 2$}. The PME belongs to the class of diffusion equations, with diffusion coefficient $D(\rho)=m\rho^{m-1}$. Since $D(\rho)$ vanishes as $\rho \to 0$, the PME is not parabolic, and its solutions can be compactly supported at each fixed time, the boundary of the positivity set $\{\rho>0\}$ moving at finite speed.  
Another important feature is that if the initial condition $\rho^{\rm ini}$ of \eqref{eq:pmeintro} is allowed to vanish, then the solution $\rho(t,u)$ can have gradient discontinuities across the interfaces which separate the positivity set $\{\rho >0\}$ from its complement. 
We refer to the monograph \cite{vaz} for an extended presentation of the mathematical properties of the PME.

\bigskip

We consider in this paper the particle system introduced in \cite{GLT}. Let us describe it in the case $m=2$ (see \eqref{eq:LN} for the general definition). The setting is one-dimensional and periodic: particles are distributed on the points of the finite torus of size $N$ denoted by $\T_N=\bb Z/N\bb Z$. We impose the exclusion restriction: no two particles can occupy the same site. 
A particle at $x$ jumps to an empty neighboring site, say $x+1$, at rate $2$ if there are particles at $x-1$ and $x+2$, at rate $1$ if there is only one particle in $\{x-1,x+2\}$, and rate $0$ else. The jump rate from $x+1$ to $x$ is given by the same rule.

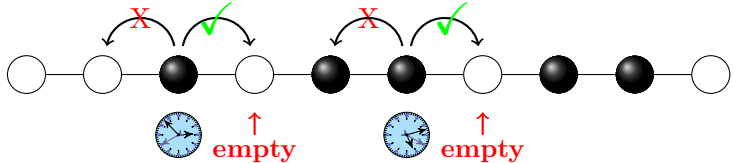
\begin{figure}
\centering
\begin{tikzpicture}
\draw[color=white] (3,-1) -- (3,2);
\centerarc[thick,<-](2.5,0.3)(10:170:0.45);
\centerarc[thick,<-](5.5,0.3)(10:170:0.45);
\centerarc[thick,->](1.5,0.3)(10:170:0.45);
\draw (1.5,0.75) node[thick,red] {X};
\draw (2.5,0.75) node[thick, green] {\LARGE\checkmark};
\draw (4.5,0.75) node[thick,red] {X};
\draw (5.6,0.75) node[thick, green] {\LARGE\checkmark};
\centerarc[thick,->](4.5,0.3)(10:170:0.45);

\draw (0,0) -- (9,0);

\shade[ball color=black](4,0) circle (0.25);
\shade[ball color=black](2,0) circle (0.25);
\shade[ball color=black](5,0) circle (0.25);
\shade[ball color=black](8,0) circle (0.25);
\shade[ball color=black](7,0) circle (0.25);

\filldraw[fill=white, draw=black]
(0,0) circle (.25)
(1,0) circle (.25)
(3,0) circle (.25)
(6,0) circle (.25)
(9,0) circle (.25);

\draw[thick,->,red] (3,-0.8) -- (3,-0.5);
\draw (3,-1.3) node[anchor=south] {\small \color{red} \bf empty};

\draw[thick,->,red] (6,-0.8) -- (6,-0.5);
\draw (6,-1.3) node[anchor=south] {\small \color{red} \bf empty};

\Clock[shift={(2,-0.8)}]{0.3}{2}{52}{40};
\Clock[shift={(5,-0.8)}]{0.3}{5}{12}{20};

\end{tikzpicture}
\caption{Allowed jumps are denoted by {\color{green}$\checkmark$}. Forbidden jumps are denoted by {\color{red}X}.}
\label{fig:jumps}
\end{figure}

As explained in \cite{GLT}, this constrained exclusion process permits to derive the PME \eqref{eq:pmeintro} with $m=2$, when the process is accelerated in the diffusive time scale $tN^2$. However, in that paper the authors need to assume that the initial profile $\rho^{\rm ini}$ is uniformly bounded away from 0 and 1, namely that it satisfies an ellipticity condition of the form $0<c_- \leqslant \rho^{\rm ini} \leqslant c_+ < 1$. With this assumption, the PME is uniformly parabolic and in particular does not display its more interesting features: finite speed of propagation and gradient discontinuities. The authors in \cite{GLT} manage to circle around the problem by perturbing the microscopic dynamics with a slowed SSEP. This way, they gain ergodicity of the Markov process and can derive the PME using the well-known \emph{entropy method} introduced in \cite{GPV}. 

In this paper, we do not assume the ellipticity condition on $\rho^{\rm ini}$ and we keep the original model described above. We believe this is the first derivation of a moving boundary problem from a conservative and degenerate microscopic dynamics (see \cite{free1,free2} for derivations in non-conservative or non-degenerate settings). One advantage of keeping the original degenerate dynamics is that one may think about studying the boundary of the positive set microscopically, even if the definition of the \emph{microscopic boundary} is absolutely not obvious and would need precaution. The microscopic behaviour of that moving interface (such that its speed, or fluctuation, for instance), as well as the relationship between the microscopic boundary and the macroscopic boundary, would be very interesting future works.

 Our choice of initial condition makes the entropy method and the \emph{relative entropy method} fail (these techniques are explained in detail in \cite{kl}). Indeed, the lack of ergodicity breaks any hope to use the entropy method and the special features of the PME are a serious obstacle to using the relative entropy method. Let us explain why and describe how we manage anyway here.

The relative entropy method was introduced for the first time by Yau \cite{yau}, and its main idea is the following: since the particle system has a family of product invariant measures indexed by the density (here, the Bernoulli product measure $\nu_\rho^N$), one can use the non-homogeneous product measures $\nu_{\rho(t,u)}^N$ with slowly varying parameter associated with the solution $\rho(t,u)$ to \eqref{eq:pmeintro}, and compare it to the state at macroscopic time $t$ of the diffusively accelerated Markov process. The latter is denoted below by $\mu_t^N$, it is a probability law on $\{0,1\}^{\T_N}$. If one expects the PME to be the correct hydrodynamic equation, these two measures should be close, and this can be seen from the investigation of the time evolution of the relative entropy $H(\mu_t^N | \nu_{\rho(t,u)}^N)$ \cm{(see \eqref{eq:hn} for the definition)}. 

In our case, two obstacles appear straight away. The first one is that $\rho(t,u)$ can take values $0$ and $1$, and therefore the above relative entropy will generally be infinite. Indeed, Yau designed this method for the linear heat equation $\partial_t \rho = \Delta \rho$ which can be wisely rewritten as $\partial_t \rho = \partial_u ( \rho(1-\rho) \partial_u (f(\rho)))$, with $f(\rho)=\log(\rho/(1-\rho))$ being  the macroscopic entropy \cite{GL}. Therefore, the application of Yau's method to the derivation of the PME is also based on the reformulation of $\partial_t \rho = \partial_{uu}(\rho^m)$ into $\partial_t \rho = \partial_{u} (m\rho^m(1-\rho) \partial_u (f(\rho)))$, with the same function $f$ which degenerates at $\rho=0$ and $\rho=1$. We note however that $f$ is \textit{not} the natural physical entropy for the PME, as explained in \cite{Otto}. The second one is that the solution $\rho(t,u)$ has poor analytic properties as soon as $\rho^{\rm ini}$ vanishes, which will complicate the control of the time evolution of the entropy. To remove these obstacles, we modify the original investigation by considering an approximation of $\rho(t,u)$, denoting ahead by $\rho_N(t,u),$ which satisfies two important properties: \begin{enumerate}[(i)]
\item it is bounded away from 0 and 1 and regular;
\item the sequence $(\rho_N)$ uniformly converges to $\rho$ on compactly supported time intervals.
\end{enumerate}
As we will see in the text, these two properties are not enough to apply straightforwardly Yau's method: we also need sharp controls on several derivatives of $\rho$. Moreover, the usual \emph{one-block estimate} (which is at the core of the relative entropy method) requires understanding the interface between the positivity set of $\rho$ and its complement, and needs very refined additional arguments. These are the main ingredients of our proof.

Let us note that the relative entropy is a tool that has been widely used in various contexts in the past forty years. 
Without being exhaustive, let us list a few applications of the relative entropy in the study of PDEs. 
It was introduced simultaneously by DiPerna \cite{Dip} and Dafermos \cite{Daf} in order to show a 
weak-strong uniqueness principle for the entropy solutions to nonlinear hyperbolic systems of conservation laws. 
It has then been used in  \cite{CarTos,Otto} to quantify the convergence 
of the solutions to the porous medium equation (set on the whole space $\R^d$) towards Barenblatt (or ZKB) self-similar profiles. 
Furthermore, it has been one of the fundamental tools in the derivation of hydrodynamic limits from Boltzmann equation 
\cite{BGL, GSR, StR}. In \cite{StR}, the author manages to extend previous results by considering an approximation 
of the solution instead of the true solution, in the same spirit as what we are doing here.
It was also used to justify rigorously reduced model obtained by asymptotic limits, like for instance in \cite{LT} 
where the relative entropy method was used to study the long-time diffusive regime of hyperbolic systems with stiff relaxation, 
or in~\cite{MN} where compressible flows in thin domains were studied. 
Finally, very recent works \cite{JR, CMS, GHMN} make use of the relative entropy in order to get error estimates for numerical approximation of PDEs.

Up to our knowledge, the present contribution is the first application of the relative entropy method to derive a hydrodynamic limit with degenerate intervals and without smoothness of the solution. To overcome that difficulty, we need to use an approximation of the solution to the hydrodynamic equation, instead of the true solution, which is the main novelty of this work.
Finally, note that the idea of plugging an approximation of the solution into the relative entropy method should certainly apply to other degenerate particle systems and allow to derive other degenerate parabolic equations. The additional work with respect to what we present here would be to derive the corresponding analytic estimates on the solution to the \cm{macroscopic equation} (see mainly Proposition~\ref{lem:pos} and the estimates in Section~\ref{ssec:properties}). The complexity of this program in higher dimensions is the reason we kept $d=1$.
 
\bigskip

Here follows an outline of the paper.  In Section \ref{sec:replacement}, we introduce and define the model with its notations, and we state our hydrodynamic limit result (Theorem \ref{theo:main}).
In Section \ref{sec:porous}, we start with recalling some specificities of the solutions to the porous medium equation, then we give a crucial property of the boundary of the positivity set. 
We also define an approximation of the solution $\rho_N$ and study its convergence. 
Finally we expose the strategy of the proof of the hydrodynamic limit through the control of $H(\mu_t^N | \nu_{\rho_N(t,u)}^N)$, which generalizes the usual relative entropy method. 
The estimates that we need about the derivatives of $\rho_N$ are proved in Section \ref{ssec:properties}. The proof of the hydrodynamic limit, and in particular the one-block estimate, is completed in Section \ref{sec:relative}.

\section{Hydrodynamics limits}\label{sec:replacement}

\subsection{Context}\label{sec:context}
Let us introduce with more details the microscopic dynamics which was first given in \cite{GLT}, and which we described in the introduction in the case $m=2$. For any $x\in\T_N$, we set $\eta(x)=1$ if $x$ is occupied, and $\eta(x)=0$ if $x$ is empty, which makes our state space $\{0,1\}^{\T_N}$.
The dynamics can be entirely encoded by the infinitesimal generator $\mathcal{L}_N$ which acts on  functions $f:\{0,1\}^{\T_N} \to \R$ as 
\begin{equation} \label{eq:LN}
\mathcal{L}_Nf(\eta):=\sum_{\substack{x,y\in\T_N\\|x-y|=1}}r_{x,y}(\eta)\eta(x)(1-\eta(y))\big(f(\eta^{x,y})-f(\eta)\big),
\end{equation}
where  \[r_{x,x+1}(\eta)=r_{x+1,x}(\eta)=\sum_{y=x-m+1}^x \prod_{\substack{z=y\\{z\notin\{ x,x+1\}}}}^{y+m}\eta(z),\]
and 
\[ \eta^{x,y}(z) = \begin{cases} \eta(y) & \text{ if } z=x, \\
\eta(x) & \text{ if } z = y, \\ 
\eta(z) & \text{ otherwise}.
\end{cases} \] 
For instance, when $m=2$, the jump rate reads 
\[ r_{x,x+1}(\eta) = \eta(x-1)+\eta(x+2),\]
and when $m=3$ it reads
\[ r_{x,x+1}(\eta)= \eta(x-2)\eta(x-1)+\eta(x-1)\eta(x+2)+\eta(x+2)\eta(x+3). 
\]
The initial configuration is random, distributed according to some initial probability measure $\mu_0^N$ on $\{0,1\}^{\T_N}$. We denote by $(\eta_t^N)_{t \geqslant 0}$ the Markov process generated by $N^2 \mathcal{L}_N$ (note that it is equivalent to accelerate time by a factor $N^2$) and starting from the initial state $\mu_0^N$. For any fixed $t\geqslant 0$, the probability law of $\{\eta_t^N(x)\; ; \; x\in\T_N\}$ on the state space $\{0,1\}^{\T_N}$ is denoted by $\mu_t^N$.

In the following we also denote by $\P_{\mu_0^N}$ the probability measure on the space of trajectories $\mathcal{D}(\R_+,\{0,1\}^{\T_N})$ induced by the initial state $\mu_0^N$ and the accelerated Markov process  $(\eta_t^N)_{t\geqslant 0}$. Its corresponding expectation is denoted by $\E_{\mu_0^N}$. 

\subsection{Product Bernoulli measures}

For any $\alpha \in [0,1]$, let $\nu_\alpha^N$ be the Bernoulli product measure on $\{0,1\}^{\T_N}$ with marginal at site $x \in\T_N$ given by
\[
\nu_\alpha^N\big\{\eta \; : \; \eta(x)=1\big\}=\alpha.
\]
In other words, we put a particle at each site $x$ with probability $\alpha$, independently of the other sites. Similarly, we define $\nu_\alpha$ as the Bernoulli product measure on $\{0,1\}^{\bb Z}$. 
We denote by $E_\alpha$ the expectation with respect to $\nu_\alpha$, and note that $E_\alpha[\eta(0)]=\alpha$. One can easily check \cm{(using the fact that $r_{x,x+1}=r_{x+1,x}$ for any $x$)} that the product measures $\{\nu_\alpha^N \; ; \; \alpha \in [0,1]\}$ are \emph{reversible} for the Markov process $(\eta_t^N)$. 

As the size $N$ of the system goes to infinity, the discrete torus $\T_N$ tends to the full lattice $\Z$. Therefore, we will need to consider functions on the space $\{0,1\}^{\Z}$. Let  $\varphi:\{0,1\}^{\Z}\to \bb R$ be a \emph{local} function, in the sense that $\varphi(\eta)$ depends on $\eta$ only through a finite number of coordinates, and therefore $\varphi$ is necessarily bounded. We then denote by $\overline \varphi(\alpha)$ its average with respect to the measure $\nu_\alpha$: \[
\overline \varphi(\alpha):=E_\alpha[\varphi(\eta)].
\]
Note that $\alpha  \mapsto \overline \varphi(\alpha)$ is continuous for every local function $\varphi$.

The one-dimensional continuous torus is denoted by $\T=\R / \Z$. Let us now define the \emph{non-homogeneous product measure} $\nu^N_{\rho(\cdot)}$ on $\{0,1\}^{\T_N}$ associated with a density profile $\rho:\T\to [0,1]$, whose marginal at site $x\in\T_N$ is given by 
\begin{equation}\label{eq:nurhoN}
\nu_{\rho(\cdot)}^N\big\{ \eta\; : \; \eta(x)=1 \big\} = 1 - \nu_{\rho(\cdot)}^N\big\{ \eta\; : \; \eta(x)=0 \big\} = \rho\big(\tfrac x N \big).  
\end{equation}
We  denote by $\E^N_{\rho(\cdot)}$ the expectation with respect to $\nu_{\rho(\cdot)}^N$. If $\rho(\cdot)$ is continuous on $\T$ and if $\varphi:\{0,1\}^\Z \to \R$ is local, then the following Riemann convergence holds:
\begin{equation}\label{eq:rieman}\frac1N \sum_{x\in\T_N} \E_{\rho(\cdot)}^N\big[ \tau_x\varphi(\eta) \big] \xrightarrow[N\to\infty]{} \int_\T E_{\rho(u)} \big[\varphi(\eta)\big] \; \d u=\int_{\T} \overline\varphi\big(\rho(u)\big)\; \d u.\end{equation}
Moreover, if a sequence of continuous profiles $\rho_N(\cdot)$ converges uniformly to $\rho(\cdot)$ on $\T$, then 
\begin{equation}\label{eq:rieman-unif}\frac1N \sum_{x\in\T_N} \E_{\rho_N(\cdot)}^N\big[ \tau_x\varphi(\eta) \big] \xrightarrow[N\to\infty]{} \int_\T E_{\rho(u)} \big[\varphi(\eta)\big] \; \d u.\end{equation}
The last convergence property will be used several times in the paper.

\subsection{Statement of the main result}

Let $\rho^{\rm ini}\in L^\infty(\T;[0,1])$ be an initial density profile. 
Our goal is to consider the hydrodynamic limit of the microscopic dynamics described in Section~\ref{sec:context}. 
As already pointed out by Gon\c{c}alves {\em et al.}~\cite{GLT}, the underlying macroscopic equation is expected to be the {\em porous medium equation} (PME)
\begin{equation}\label{eq:PME}
\begin{cases}
\p_t \rho = \p_{uu}(\rho^m) & \text{in}\; \R_+\times \T, \\
\rho_{|_{t=0}}=\rho^{\rm ini}&\text{in}\; \T. 
\end{cases}
\end{equation}
This equation is of degenerate parabolic type. 
It is well known that the notion of strong solution ---i.e., $\rho \in \mathcal{C}^{1,2}(\R_+ \times \T)$---
is not suitable to get the well-posedness of the problem~\eqref{eq:PME} 
unless $\rho^{\rm ini}$ remains bounded away from $0$. Indeed, the space derivative 
of $\rho$ may be discontinuous at the boundary of the set $\{\rho>0\}$ (see for instance~\cite{vaz}).
This motivates the introduction of the following notion of weak solutions.
\begin{definition}\label{def:weak_sol}
A function $\rho \in L^\infty(\R_+ \times \T; [0,1])$ is said to be a weak solution to~\eqref{eq:PME} 
corresponding to the initial profile $\rho^{\rm ini}$ if  $\p_u (\rho^m) \in L^2(\R_+\times\T)$ 
and
\be\label{eq:weak}
\iint_{\R_+ \times \T}\rho \; \p_t \xi\; \d u \d t + \int_\T \rho^{\rm ini} \;\xi(0,\cdot) \d u - \iint_{\R_+ \times \T} \p_u (\rho^m) \p_u \xi \; \d u \d t = 0, \quad 
\text{for all } \xi \in C^1_c(\R_+ \times \T).
\ee
\end{definition}

What we call a weak solution corresponds to what is called an energy solution in Vazquez' monograph (see \cite[Section 5.3.2]{vaz}).
The classical existence theory based on compactness arguments (see for instance \cite[Therorem 5.5]{vaz}) can be extended to our periodic 
setting without any difficulty. The uniqueness of the weak solution  and the fact that they remain bounded between $0$ and $1$ are consequences of the following 
$L^1$-contraction/comparison principle~(see \cite[Proposition 6.1]{vaz}): let $\rho^{\rm ini}$ and $\check \rho^{\rm ini}$ be two initial profiles in $L^\infty(\T;[0,1])$, and let $\rho$ 
and $\check \rho$ be corresponding weak solutions, then
\begin{equation}\label{eq:comparison}
\int_\T (\rho(t,u) - \check \rho(t,u))^+ \d u \leq\int_\T (\rho^{\rm ini}(u) - \check \rho^{\rm ini}(u))^+ \d u, \qquad \text{for any } t \ge0, 
\end{equation}
where $a^+ = \max(a,0)$ denotes the positive part of $a$. In the above relation, we have used the fact that any weak solution to~\eqref{eq:PME}
belongs to $\mc C(\R_+;L^1(\T))$ (see for instance~\cite{CG11}). 

\clem{
As it will appear in the sequel, the so-called {\em pressure}, denoted by ${\Pr}$ in what follows, plays an important role. It is related to the 
density $\rho$ by the monotone relation \[{\Pr} = {\Pr}(\rho)= \frac{m}{m-1} \rho^{m-1}.\] The equation~\eqref{eq:PME} then rewrites 
\[
\p_t \rho - \p_u \left(\rho \p_u {\Pr}\right) = 0.
\]
We denote ${\Pr}^{\text{ini}}= \frac{m}{m-1} \left(\rho^{\text{ini}}\right)^{m-1}$, and} we now state our assumption on the initial condition.
 \begin{assumption}[The initial profile] \label{ass:ini} We assume that:
\begin{itemize}
\item  The initial pressure profile ${\Pr}^{\rm ini}$ is Lipschitz continuous, namely there exists $C_{\rm Lip}>0$ such that 
\begin{equation} \big\| \partial_u {\Pr}^{\rm ini}\big\|_{\infty} \leqslant C_{\rm Lip},\label{eq:lip}\end{equation}
where $\|\cdot \|_\infty$ denotes the usual $L^\infty$-norm;
\item The initial positivity set 
\begin{equation}
\label{eq:Pini} \mathcal{P}_{0}: = \big\{u \in\T \; : \; \rho^{\rm ini}(u)>0\big\}
\end{equation} 
has a finite number of connected components.
\end{itemize}
\end{assumption}
Note that this assumption is less restrictive than the one given in \cite{GLT}, where $\rho^{\rm ini}$ was supposed to be uniformly bounded away from 0 and 1. In particular, we authorize vanishing initial profiles. Our main result reads as follows:

 \begin{theorem}\label{theo:main}
 We assume that the initial microscopic configuration $\{\eta_0(x)\; : \; x \in\T_N\}$ is distributed according to $\mu_0^N= \nu^N_{\rho^{\rm ini}(\cdot)}$, \cm{with $\rho^{\rm ini}$ satisfying Assumption \ref{ass:ini}}.
 
 Then, the following \emph{local equilibrium convergence} holds at any macroscopic time $t>0$: for any continuous function $G:\T \to \R$, any local function $\varphi:\{0,1\}^\Z \to \R$
 \begin{equation}\lim_{N\to\infty}\E_{\mu_0^N} \bigg[\bigg|\frac1N \sum_{x\in\T_N}G\Big(\frac x N\Big)\tau_x\varphi\big(\eta_t^N\big) - \int_{\T} G(u)\overline\varphi\big(\rho(t,u)\big) \d u \bigg|\bigg] =0,  \label{eq:localequil}\end{equation}
 where $\rho$ is the unique weak solution of \eqref{eq:PME} in the sense of Definition~\ref{def:weak_sol}. 

\end{theorem}

\begin{remark}\label{Barenblatt}
The porous medium equation \eqref{eq:PME} admits fundamental solutions, which are usually called Barenblatt (or ZKB) solutions. An explicit form for the Barenblatt solution is
\[
\rho^{B}(t,x)=t^{-\frac{1}{m+1}}\left( \left( C- \frac{(m-1)x^2}{2m(m+1)t^{\frac{2}{m+1}}}\right)^+\right)^{\frac{1}{m-1}}
\]
for each $C>0$. In particular, for sufficiently small $t>0$, its support is contained in $\T$ and $\rho^{ini}(\cdot):=\rho^B(t,\cdot)$ satisfies Assumption \ref{ass:ini}.
\end{remark}

\section{Strategy of the proof}\label{sec:porous}
 
 Let us give here some properties of the solution to the PME to be used in the sequel. 
 Sometimes we prove the results only partially, and we invite the reader to check the details 
 of the proofs in the monograph  \cite{vaz}  written by J.L.Vazquez. 
 Precise references will be given for each result. 
 
 If the porous medium equation starts with an initial profile which vanishes, then the solution at any later time can have discontinuous gradients across the \emph{interfaces} at which the function becomes positive. This is a problem when one tries to prove  hydrodynamic limits.
 The best way to tackle discontinuity problems is to  slightly perturb the initial condition, by making it positive, and bounded away from 1. 

In Section \ref{ssec:pme-orig}, we state some properties of the PME starting from an initial profile which can lead to singularities at positive times. In Section \ref{ssec:pme-modif} we modify the initial condition so as to regularize the solution of the PME and gain better control estimates. In Section \ref{ssec:strategy} we expose the strategy to prove Theorem \ref{theo:main}. 

In the following we denote by  $\|\cdot\|_p$ the usual $L^p$-norm, whenever the integration spaces are clear to the reader. Otherwise, the $L^p(\Omega)$-norm will be denoted by $\|\cdot\|_{L^p(\Omega)}$.

\subsection{The porous medium equation (PME)} \label{ssec:pme-orig}

 We start with recalling some properties of the unique weak solution $\rho(t,u)$ to \eqref{eq:PME}. 
Our first statement is related to the continuity of the weak solutions to the porous medium equation. 
 Such a regularity result can be deduced from \cite[Section 7.7 and Section 15.1]{vaz}. It is also a straightforward consequence of the forthcoming Proposition \ref{prop:unif}.

 \begin{proposition}[Regularity of the solution] \label{prop:regul-ast} The unique weak solution to \eqref{eq:PME} is continuous on $\R_+\times \T$, \cm{and the corresponding pressure $\Pr=\frac{m}{m-1}\rho^{m-1}$ is Lipschitz continuous.} 

 \end{proposition}

\def\Pp{\mathcal{P}}

Let us denote by $\mathring{A}$ the interior of the subset $A\subset \mathbb{T}$ and by $\overline{A}$ its closure. For all $t\geqslant 0$, we denote by 
$$
\mathcal{P}_t := \big\{u \in \bbT\; : \;\rho(t,u) >0 \big\}
$$
the positivity set of  $\rho(t,\cdot)$, which is an open subset of $\bbT$ since $\rho(t,\cdot)$ is continuous.
Finally we denote by 
\begin{equation}\label{eq:gamma}
\Gamma_t := \partial\mathcal{P}_t = \overline{\mathcal{P}_t} \setminus \mathcal{P}_t
\end{equation}
the {\em interface} between the  positivity set $\mathcal{P}_t$ of $\rho(t,\cdot)$ and the complementary 
\begin{equation}\label{eq:Z}\mathcal{Z}_t: = \overbrace{\left\{u \in \bbT\; : \; \rho(t,u) =0 \right\}}^{\circ} = \bbT \setminus \overline{\mathcal{P}_t}\end{equation}
of its support. 
Note that $\Gamma_t$ is closed, and is a nowhere dense set, but it can \emph{a priori} have positive Lebesgue measure. Actually, we will prove in Lemma \ref{lem:connect} below that from our assumption \eqref{eq:Pini} on $\mc P_{0}$, this does not happen and that the Lebesgue measure of $\Gamma_t$ vanishes for any $t>0$.
 \begin{remark} 
Let us underline that the derivatives of the pressure $\Pr$ can have jump discontinuities on the so-called \emph{free boundary}
$$\bigcup_{t\in(0,T]} \{t\}\times \Gamma_t,$$
but both $\Pr$ and $\rho$ are smooth outside of this set at positive times.  We refer the reader to \cite[Chapter 14]{vaz} for the general theory and also \cite[Chapter 4]{vaz} for several examples.
 \end{remark}
 
In what follows, the notation $\mathrm{Leb}$ stands for the usual Lebesgue measure restricted on $\mathbb{T}$, and  
$|B|$ denotes the cardinality of the discrete subset $B\subset \T_N$.

\begin{proposition}[Positivity intervals]\label{lem:pos}
For any $\delta >0$ and $t \in [0,T]$ we set
\begin{equation}\label{eq:gammadelta}
\Gamma_t(\delta) = \overline{\Big\{ u\in\T \; :\; 0 < \rho(t,u) <\delta\Big\}}.
\end{equation}
We have
\be\label{eq:interf}
\int_0^T \mathrm{Leb} \left(\Gamma_t(\delta)\right) \d t \xrightarrow[\delta\to 0]{}0.
\ee
 \end{proposition}
 
 \begin{proof}[Proof of Proposition \ref{lem:pos}] The proof follows from the following technical lemma, which we will prove ahead in Appendix \ref{app:profile}.
 \begin{lemma}[Connected components of the positivity set]\label{lem:connect}
For any $t>0$, $\mc P_t$ has a finite number of connected components.
\end{lemma}
From last lemma, since $\mc P_t$ has a finite number of connected components for any $t>0$, we know that $\Gamma_t$ is a finite union of points,
and therefore
 ${\rm Leb}(\Gamma_t) = 0$.
Since 
\[
\Gamma_t = \bigcap_{\delta >0}\overline{\Big\{ u\in\T \; :\; 0 < \rho(t,u) <\delta\Big\}} = \bigcap_{\delta >0} \Gamma_t(\delta), 
\]
it follows from the monotonicity of the Lebesgue measure that 
\[
0 = {\rm Leb} (\Gamma_t) = \lim_{\delta \to 0} {\rm Leb}(\Gamma_t(\delta)), \qquad \text{for any } t \in [0,T].
\]
Moreover, since $\Gamma_t(\delta) \subset \bbT$, we get that ${\rm Leb}(\Gamma_t(\delta)) \leq 1$ for all $t \in [0,T]$.
Hence \eqref{eq:interf} follows from Lebesgue's dominated convergence Theorem.
\end{proof}

\subsection{The regularized initial condition} \label{ssec:pme-modif}
In order to prove Theorem \ref{theo:main}, we need to introduce a \emph{regularized approximate solution} to the PME. This is the goal of this section.
 
Let $(\varepsilon_N)_{N\in\bb N}$ be a vanishing sequence such that $ \varepsilon_N \in (0,\frac12)$. The \clem{rate} at which $\varepsilon_N \to 0$ will be made more precise later on. 
Let $h\in \mc C^\infty(\R)$ be such that $h \geqslant 0$ and 
\begin{align}
(i) &\quad \mathrm{Supp}(h) \subset (-1,1), \qquad \int_\R h(y)\d y=1, \label{eq:h1}\\
(ii) &\quad h(y)=h(-y), \quad \text{ for any } y \in \R \vphantom{\int} \label{eq:h2}\\
(iii) &\quad \partial_yh(y) \leqslant 0, \quad \text{ if } y \geqslant 0. \label{eq:h3}
\end{align}
Denote $C_h:=\|h\|_\infty$. It follows from \eqref{eq:h2} and \eqref{eq:h3} that $\|\partial_y h\|_1 = 2 C_h$. Let us define
the \emph{regularizing approximation of the unit}:
\[h_N(y) = \varepsilon_N^{-1}\; h(\varepsilon_N^{-1}y), \qquad y \in \R\] which satisfies ${\rm Supp}(h_N)\subset(-\varepsilon_N,\varepsilon_N)$, and also
\begin{equation} 
\big\|h_N\big\|_1 = 1, \qquad \big\|h_N\big\|_\infty = \frac{C_h}{\varepsilon_N}, \qquad \big\|\partial_yh_N\big\|_1 = \frac{2C_h}{\varepsilon_N}.\label{eq:h_n}
\end{equation}
From here several steps are necessary to define the approximate initial data $\rho_N^\text{ini}$. First, we introduce the truncated 
initial density and pressure defined by
 \begin{align*}
 \widetilde \rho_N^{\rm ini} =& \max\big\{\varepsilon_N, \; \min \big(1-\varepsilon_N,\; \rho^{\rm ini}\big)\big\}, 
\\
 \widetilde{{\Pr}}_N^{\rm ini} = &  \frac{m}{m-1} \left(\widetilde \rho_N^{\rm ini}\right)^{m-1} 
 =  \max\bigg\{\frac{m}{m-1}\varepsilon_N^{m-1}, \; \min \Big(\frac{m}{m-1}(1-\varepsilon_N)^{m-1},\; {\Pr}^{\rm ini}\Big)\bigg\}.
 \end{align*}
The \emph{truncated and regularized initial data} are then defined by
\begin{equation}\label{eq:rho-ini}
{\Pr}_N^{\rm ini} = \widetilde{{\Pr}}_N^{\rm ini} \star h_N, \qquad 
\rho_N^{\rm ini} = \left(\frac{m-1}m {\Pr}_N^{\rm ini}\right)^{\frac1{m-1}},
\end{equation}
where $\star$ is the usual convolution product on $\T$. 
This approximation procedure is designed so that the following properties hold:
\begin{enumerate}
\item \textit{Regularity}: $\rho_N^{\rm ini}$ and ${\Pr}_N^{\rm ini}$ are smooth on $\T$;
\item \textit{Boundedness away from $0$ and $1$}: 
 \begin{equation}\label{eq:ini-eps}
 \varepsilon_N \leqslant \rho_N^{\rm ini} \leqslant 1-\varepsilon_N, 
\end{equation}
\item \textit{Lipschitz regularity of the regularized pressure}:
\begin{equation}\label{eq:Lip-ini-reg}
\left\|\p_u {\Pr}_N^{\rm ini}\right\|_\infty \leq C_{\rm Lip},
\end{equation}
where $C_{\rm Lip}$ has been introduced in \eqref{eq:lip}.
\item \textit{Uniform convergence towards the initial profiles}:
\begin{align}
 \left\| {\Pr}_N^{\rm ini} - {\Pr}^{\rm ini} \right\|_\infty &\leq (m+C_{\rm Lip}) \eps_N \xrightarrow[N\to\infty]{} 0
\label{eq:conv_pini}
\\
\big\|\rho_N^{\rm ini} - \rho^{\rm ini} \big\|_\infty &\leqslant C_{\rm ini} \left(\varepsilon_N\right)^{\frac1{m-1}} \xrightarrow[N\to\infty]{} 0. \label{eq:conv}\end{align}
with $C_{\rm ini} = \left( \frac{(m-1)}{m}(m+C_{\rm Lip})\right)^{\frac{1}{m-1}}$. 
\end{enumerate}
Note that \eqref{eq:Lip-ini-reg} and the definition \eqref{eq:rho-ini} imply: 
\begin{equation} \label{eq:Lip-rho}
\left\|\partial_u \rho_N^{\rm ini} \right\|_\infty \leqslant \frac{C_{\rm Lip}}{m} (\varepsilon_N)^{2-m},
\end{equation} therefore $\rho_N^{\rm ini}$ is uniformly Lipschitz only in the case $m=2$. If $m \geqslant 3$ the right hand side above goes to infinity as $N \to \infty$.

\medskip

Let us now define the \emph{regularized solution} $\rho_N$ on $\bb R_+ \times \T$ as the solution to 
\begin{equation}
\label{eq:porous-smooth}
  \begin{cases} \partial_t \rho_N  = \partial_{uu} \big( (\rho_N)^m\big) & \text{in } \R_+\times \T, \vphantom{\bigg(} \\ 
(\rho_N)_{|_{t=0}} = \rho_N^{\rm ini} & \text{in } \T.\end{cases}
\end{equation}
This solution will play a central role in the proof of Theorem \ref{theo:main}, \clem{as well as the corresponding regularized pressure: 
\begin{equation}\label{eq:pressure-reg}
{\Pr}_N = \frac{m}{m-1} (\rho_N)^{m-1}.
\end{equation}} Let start here with two major properties of $\rho_N$.

\begin{proposition}\label{prop:max} Fix a time horizon line $T>0$.
Problem \eqref{eq:porous-smooth} admits a unique strong solution $\rho_N \in \mc C^\infty([0,T] \times \T)$ which satisfies
\be\label{eq:max-eps}
\varepsilon_N \leqslant \rho_N \leqslant 1-\varepsilon_N.
\ee
\end{proposition}

\begin{proof}[Proof of Proposition \ref{prop:max}]
The uniqueness of the weak (then strong) solution follows from the monotonicity of the porous medium equation, which yields 
$L^1$-contraction and a comparison principle (see for instance~\cite{CT80}). It follows from this comparison 
principle that $\varepsilon_N \leqslant \rho_{N} \leqslant 1-\varepsilon_N$ a.e.~in $[0,T]\times\T$. Therefore, the solution remains 
bounded away from the degeneracy $\rho=0$ of the PME \eqref{eq:PME}. The problem~\eqref{eq:porous-smooth} 
is then uniformly parabolic. It follows from the classical regularity theory for parabolic equations 
(see for instance~\cite{LSU}) that $\rho_N$ is smooth. See also \cite[Theorem 3.1, Proposition 12.13]{vaz}. 
\end{proof}
 
\begin{proposition}[Uniform convergence]\label{prop:unif}
The sequence $(\rho_N)_{N\in\N}$ converges uniformly in $[0,T]\times \T$ towards the unique weak solution to \eqref{eq:PME}. 
\end{proposition} 

\begin{proof}[Proof of Proposition \ref{prop:unif}]
It follows from the comparison principle \eqref{eq:comparison} that 
\[
\int_\bbT |\rho(t,u) - \rho_N(t,u)|\d u \le \int_\bbT |\rho^\text{ini}(u) - \rho_N^\text{ini}(u)|\d u, \qquad \text{for any } \; t \in [0,T]. 
\]
Hence, we deduce from estimate~\eqref{eq:conv} that $(\rho_N)_{N\in\bb N}$ converges in $\Cc([0,T];L^1(\bbT))$ towards $\rho$.
Therefore, it suffices to show that $\left(\rho_N\right)_{N\in \bb N}$ is relatively compact in $\Cc( [0,T]\times \bbT)$ to conclude the proof 
of Proposition~\ref{prop:unif} thanks to the uniqueness of the limit value. Our proof mainly follows the program of \cite[Section 7.7]{vaz}. 
We first need to introduce the \emph{fractional Sobolev spaces} $H^s(\T)$. 
We refer to~\cite{H2GHs} for an overview on fractional Sobolev spaces.
Since we are in the simple situation where the domain is the one-dimensional torus, 
such spaces are very easy to define and to manipulate with Fourier series. 
We recall here its core properties to be used in what follows.
Given $s\in[0,1]$, a function $\rho:\T\to\R$ belongs to $H^s(\T)$ iff 
\[\|\rho\|_{H^s(\T)}:=\bigg(\sum_{k\in\Z} \big(1+4\pi^2|k|^2\big)^s\; |\widehat{\rho}_k|^2\bigg)^\frac12 < \infty,\]
where the Fourier coefficient $\widehat\rho_k$ reads as 
$
\widehat{\rho}_k:=\int_{\T} \rho(u) e^{-i2\pi ku} \d u.
$
From Parseval's relation, we have
\[\|\rho\|_{H^1(\T)}^2 = \|\rho\|_{L^2(\T)}^2 + \|\partial_u \rho\|_{L^2(\T)}^2.\]
\clem{The space $H^s(\T)$ is compactly (hence continuously) embedded in $\mc C(\T)$ as soon as $s>\frac12$. 
Moreover, for any $\rho\in H^1(\T)$, the following interpolation inequality holds:
\begin{equation}\label{eq:interp-Hs}
\|\rho\|_{H^s(\T)} \leq \big\|\rho\big\|_{L^2(\T)}^{1-s}\;  \big\|\rho\big\|_{H^1(\T)}^s, \qquad \text{for any } s \in [0,1].
\end{equation}}
Going back to our problem, let us multiply the PME~\eqref{eq:porous-smooth} by $\p_t \left(\rho_N^\clem{m}\right)$ and then integrate over $[0,t^\star]\times\bbT$ for some 
arbitrary $t^\star \in [0,T]$ to get 
\[
A_N(t^\star)+B_N(t^\star) = 0, 
\]
where 
\[
A_N(t^\star) =  \iint_{[0,t^\star]\times \T} \p_t \rho_N\; \p_t \left(\rho_N^\clem{m}\right) \d t \d u, \qquad 
B_N(t^\star) = \iint_{[0,t^\star]\times \T} \p_u \left(\rho_N^\clem{m}\right) \; \p_{ut} \left(\rho_N^\clem{m}\right) \d t \d u. 
\]
The bound $|\rho_{N}|\le 1$ yields
\[
A_N(t^\star) \ge \frac{1}{m} \iint_{[0,t^\star]\times \T}  \left| \p_t v_N\right|^2 \d t\d u, 
\]
where we have set $v_N := \rho_N^\clem{m}$. On the other hand,
\[
B_N(t^\star) =  \frac12\int_\bbT \big|\p_u  v_N(t^\star,u)\big|^2 \d u - 
\frac12\int_\bbT \left|\p_u \left(\left(\rho_{N}^\text{ini}\right)^\clem{m}\right)\right|^2 \d u.
\]
The bound \eqref{eq:Lip-ini-reg} together with $0 \leq \rho_N^{\rm ini} \leq 1$ provide that 
\[
\frac12\int_\bbT \left|\p_u \left(\left(\rho_{N}^\text{ini}\right)^\clem{m}\right)\right|^2 \d u \clem{= \frac12\int_\bbT \Big|\rho_N^{\rm ini} \; \p_u {\Pr}_N^{\rm ini} \Big|^2 \d u } \le \clem{\frac12} (C_{\rm Lip})^2.
\]
Hence, we obtain that 
\[
{\frac{2}{m}} \iint_{[0,t^\star]\times \T}  \left| \p_t v_N \right|^2 \d t \d u + 
\int_\bbT \left|\p_u v_N(t^\star,u)\right|^2 \d u \le (C_{\rm Lip})^2, 
\qquad \text{for any } \;t^\star \in [0,T].
\]
To sum up, we have the following (uniform w.r.t.~$N$) estimates 
on the sequence $\left(v_{N}\right)_{N}$: denoting $C=C_{\rm Lip}\sqrt{m/2}$, 
\begin{align}\label{eq:u-inf}
&\big\|v_N\big\|_{L^\infty([0,T]\times\T)} \le 1, \vphantom{\Big(}
 \\ \label{eq:u_t}
&{\big\|\p_t v_N\big\|}_{L^2([0,T]\times \T)} \le C, \vphantom{\Big(}
\\ \label{eq:u_x}
\sup_{t\in[0,T]}&{\big\|\p_u v_N(t,\cdot)\big\|}_{L^2(\bbT)} \le C. \vphantom{\Big(}
\end{align}
It follows from~\eqref{eq:u_x} and the Cauchy-Schwarz inequality that 
\be\label{eq:holder_x}
|v_N(t,u) - v_N(t,\hat u )| \le C|u-\hat u|^{\frac12}, \qquad \text{for any }\;  u,\hat u \in \bbT, \;  t \in [0,T].
\ee
Similarly, we deduce from~\eqref{eq:u-inf} and~\eqref{eq:u_t} that $\left(v_N\right)_{N}$ is 
uniformly bounded in the space $\Cc^{0,\frac12}\big([0,T]; L^2(\bbT)\big)$, i.e., 
\be\label{eq:holder_t}
\big\|v_N(t) - v_N(\hat t)\big\|_{L^2(\bbT)} \le C |t - \hat t|^\frac12, 
\qquad  \text{for any }\; t,\hat t \in [0,T].
\ee
Using~\eqref{eq:interp-Hs}, 
we get that 
\[
\big\|v_N(t) - v_N(\hat t)\big\|_{H^s(\bbT)} \le {\big\|v_N(t) - v_N(\hat t)\big\|}_{H^1(\bbT)}^s  \; 
{\big\|v_N(t) - v_N(\hat t)\big\|}_{L^2(\bbT)}^{1-s}, \qquad \text{for any } \; t,\hat t \in [0,T].
\]
Combining it with~\eqref{eq:u_x} and \eqref{eq:holder_t}, this provides 
\[
\big\|v_N(t) - v_N(\hat t)\big\|_{H^s(\bbT)} \le \clem{C\left(4+T\right)^{\frac{s}2}} |t - \hat t|^{\frac{1-s}2}, \qquad\text{for any }\; t,\hat t \in [0,T].
\]
Choosing $s\in (\frac12, 1)$ and using the continuous embedding of $H^{s}(\bbT)$ in $\Cc(\bbT)$ 
we get that 
\be\label{eq:holder_t2}
|v_N(t,u) - v_N(\hat t,u)| \le  \clem{C\left(4+T\right)^{\frac{s}2}} |t - \hat t|^{\frac{1-s}2}, 
\qquad  \text{for any } \; u \in \bbT,\; t,\hat t \in [0,T].
\ee
The combination of~\eqref{eq:holder_x} with \eqref{eq:holder_t2} provides: \clem{there exists $C'>0$ that depends on $(C_{\rm Lip}, T,s)$ such that}, for any $ u, \hat u \in \bbT$, and $t,\hat t \in [0,T]$, 
\begin{align*} 
|v_N(t,u) - v_N(\hat t,\hat u )| & \le  | v_N(t,u) - v_N( \hat t,u)| + |v_N(\hat t,u) - v_N(\hat t,\hat u )| \\&
\le  
\clem{C'}
\left( |t - \hat t|^{\frac{1-s}2} + |u-\hat u|^{\frac12}\right).
\end{align*}
Therefore, one can apply Arzela-Ascoli's Theorem and claim that $\left( v_N\right)_N$ is 
relatively compact in $\Cc([0,T]\times\bbT)$, and thus so is $\left(\rho_N\right)_N = \big(\clem{(v_N)^{\frac1m}}\big)_N$. 
\end{proof}

\subsection{Relative entropy method} \label{ssec:strategy}

In the following, for any two probability measures $\mu, \nu$ on $\{0,1\}^{\T_N}$ we denote by $H(\mu|\nu)$ the relative entropy of $\mu$ with respect to $\nu$, defined as usual by
\[H(\mu | \nu) = \sup_{f} \bigg\{ \int f d\mu - \log \int e^f d\nu\bigg\},\]
where the supremum is carried over all real valued functions. The following \emph{entropy inequality} is going to be useful: for any $\gamma >0$,  we have
\begin{equation}
\label{eq:entrop}
\int f d\mu \leqslant \frac{1}{\gamma}\Big(\log \int e^{\gamma f} d\nu + H(\mu|\nu)\Big). 
\end{equation}
Recall that we denote by $\E^N_{\rho_N(t,\cdot)}$ the expectation with respect to the non-homogeneous Bernoulli product measure $\nu_{\rho_N(t,\cdot)}^N$. Fix $\alpha \in (0,1)$ and an invariant measure $\nu_\alpha$.  We introduce the density
\[
\psi_t^N(\eta):=\frac{\d\nu_{\rho_N(t,\cdot)}^N}{\d\nu_\alpha}(\eta)=\frac{1}{\mathsf{Z}_t^N} \exp\bigg(\sum_{x\in\T_N}\eta(x)\; \lambda_N\Big(t,\frac{x}{N}\Big)\bigg),
\]
where 
\be
\lambda_N(t,u)=\log\bigg(\frac{\rho_N(t,u)(1-\alpha)}{\alpha(1-\rho_N(t,u))}\bigg),
\label{eq:lambdaN}
\ee
and $\mathsf{Z}_t^N$ is the normalization constant. 
Note that $\lambda_N$ is well defined thanks to Proposition \ref{prop:max}. 
Recall moreover that $\mu_{t}^N$ is the distribution of the accelerated process at time $tN^2$ and denote its density with respect to $\nu_\alpha$ as
\[
f_t^N:=\frac{\d\mu_{t}^N}{\d\nu_\alpha}.\] Finally, we are interested in the relative entropy \begin{equation}\label{eq:hn} \mc H_N(t):=H\big(\mu_{t}^N | \nu_{\rho_N(t,\cdot)}^N\big)=\int f_t^N(\eta) \log\Big(\frac{f_t^N(\eta)}{\psi_t^N(\eta)}\Big) \d \nu_{\alpha}(\eta).
\end{equation}
The proof of Theorem \ref{theo:main} is based on the investigation of the time evolution of the entropy $\mc H_N(t)$. This strategy is inspired by the \emph{relative entropy method} which is exposed in details  for instance in \cite[Chapter 6]{kl}. However, in our case the standard method cannot work:  the usual scheme works with the relative entropy of $\mu_{t}^N$ with respect to the product measure $\nu_{\rho(t,\cdot)}^N$, associated with the \emph{true} weak solution of the PME \eqref{eq:PME}. As we have seen in Section \ref{ssec:pme-orig}, this solution has poor regularity properties, and more importantly, it can vanish on non-trivial intervals. This would make the relative entropy take infinite values for presumably long times. 

This is why we work with a different relative entropy: here, $\mc H_N(t)$ defined in \eqref{eq:hn} involves the non-homogeneous product measure $\nu^N_{\rho_N(t,\cdot)}$, which is associated with the \emph{regularized} solution $\rho_N$, defined in \eqref{eq:porous-smooth}. Since $\rho_N$ is smooth and bounded away from $0$ and $1$, the relative entropy is always finite. Since $(\rho_N)$ uniformly converges to $\rho$ on $[0,T]\times\T$ (from Proposition \ref{prop:unif}), one  might believe that the arguments of \cite{kl} can be easily adapted. However, one needs much more than uniform convergence. In particular, sharp controls on the derivatives of $\rho_N$ are also needed, as explained in the rest of the paper. 

Let us conclude this section with two important results concerning $\mc H_N(t)$. At the end of this paragraph we will show how do they imply Theorem \ref{theo:main}. First of all, at $t=0$, the initial relative entropy is of order \cm{$N(\varepsilon_N)^{\frac{1}{m-1}} |\log \varepsilon_N|$} as $N\to\infty$, namely:
\begin{lemma}[Initial entropy]\label{l:smallentropy}
\begin{equation*} \mc H_N(0)=H\big(\mu_0^N | \nu^N_{\rho_N^{\rm ini}(\cdot)}\big) = H\big(\nu_{\rho^{\rm ini}}^N|\nu^N_{\rho_N^{\rm ini}(\cdot)}\big)=\cm{\mc O\big(N(\varepsilon_N)^{\frac{1}{m-1}}|\log\varepsilon_N|\big)}=o(N),  \text{ as } N\to\infty.\end{equation*}
\end{lemma}

This lemma is proved in Section~\ref{s:smallentropy}. Next, we are able to control the entropy production on a finite time interval, thanks to all the sharp estimates that we will obtain in Section \ref{ssec:properties}. This is where we need to make an assumption on the convergence speed of $(\varepsilon_N)$: 

\begin{assumption}[Convergence speed of $\varepsilon_N$] \label{ass:eps}
\begin{equation}
\label{eq:ass-eps}
\cm{\lim_{N\to\infty} N(\varepsilon_N)^{{6m-6}} = + \infty.}
\end{equation}
\end{assumption}

\begin{proposition}[Entropy production]\label{thm:entropy}
Under Assumption \ref{ass:eps}, there exists a constant $\kappa>0$ such that 
\[
\mc H_N(T) \leqslant \kappa \int_0^T \mc H_N(s) \d s + o_T(N),
\]
where $o_T(N)$ stands for a sequence of real numbers $C(T,N)$ such that $C(T,N)/N \to 0$ as $N\to \infty$.
\end{proposition}
We prove this result in Section \ref{sec:proof}. 

From Gronwall's inequality and Lemma \ref{l:smallentropy},  we conclude:
\begin{corollary}\label{cor:entropy} For any $t >0$,
\[H\big(\mu_t^N \; | \; \nu_{\rho_N(t,\cdot)}^N\big) =\mc H_N(t) =  o_t(N), \qquad \text{as } N\to\infty .\]
\end{corollary}

Then, one has to prove that Corollary \ref{cor:entropy} is sufficient to show the local equilibrium result \eqref{eq:localequil} stated in Theorem \ref{theo:main}. To do so, one needs to know that the approximate solution $\rho_N(t,\cdot)$ converges uniformly to $\rho(t,\cdot)$ in $\T$ (which does hold from Proposition \ref{prop:unif}), and that the solution $\rho(t,\cdot)$ is continuous. We have all in hands to conclude the proof of Theorem \ref{theo:main}: 

\begin{proof}[Proof of Theorem~\ref{theo:main}] One has to compute the limit of the left hand side of \eqref{eq:localequil}. For the sake of clarity, we assume that the local function $\varphi$ only depends on the configuration value at 0, namely: $\varphi(\eta)=\varphi(\eta(0))$. Recall that we want  to prove that the expectation
\begin{equation}\label{eq:exp}\bb E_{\mu_0^N} \bigg[\bigg|\frac1N \sum_{x\in\T_N}G\Big(\frac x N\Big)\tau_x\varphi\big(\eta_t^N\big) - \int_{\T} G(u)\overline\varphi\big(\rho(t,u)\big) \d u \bigg|\bigg] \end{equation}
vanishes as $N\to\infty$.  Note that $G$ and $\rho(t,\cdot)$ are continuous and bounded. Then, for any fixed $t>0$, we easily replace 
\[\int_{\T} G(u)\overline\varphi\big(\rho(t,u)\big) \d u \quad \text{with}\quad \tfrac1N \sum_{x\in\T_N}G\big(\tfrac x N\big)\overline\varphi\big(\rho(t,\tfrac x N)\big),\] paying a small price of order $o_t(1)$. 
Next, we perform an integration by parts, and we bound as follows:
\begin{align} \int &\bigg|\tfrac1N \sum_{x\in\T_N}G\big(\tfrac x N\big)\varphi(\eta(x)) - \tfrac1N \sum_{x\in\T_N}G\big(\tfrac x N\big)\overline\varphi\big(\rho(t,\tfrac x N)\big)\bigg| \d \mu_t^N(\eta)\notag \\
& \leqslant \int  \tfrac1N \sum_{x\in\T_N} \bigg| \tfrac1{2\ell+1} \sum_{|y-x|\leqslant \ell} \big(G\big(\tfrac y N\big)-G\big(\tfrac x N\big)\big)\varphi(\eta(y))\bigg| \d \mu_t^N(\eta) \label{eq:firstlim} \\
&\quad  + \int\tfrac1N \sum_{x\in\T_N} \bigg|\tfrac{1}{2\ell+1}G\big(\tfrac x N\big) \sum_{|y-x|\leqslant \ell} \big(\varphi(\eta(y))-\overline\varphi\big(\rho(t,\tfrac x N)\big)\big)\bigg| \d \mu_t^N(\eta). \label{eq:secondlim} \end{align}
Since $G$ is smooth, the first limit \eqref{eq:firstlim} vanishes as $N\to\infty$ and then $\ell\to\infty$. Since $G$ is bounded,  \eqref{eq:exp} vanishes if we are able to prove that
\[
\limsup_{\ell \to\infty}\limsup_{N\to\infty} \int \bigg(\frac{1}{N}\sum_{x\in\T_N} \Big|\frac{1}{2\ell+1}\sum_{|y-x|\leqslant \ell} \varphi(\eta(y)) - \overline\varphi\big(\rho\big(t,\tfrac x N\big)\big)\Big|\bigg)\d \mu_t^N(\eta)=0.
\]
By the entropy inequality \eqref{eq:entrop}, for every $\gamma >0$, we bound the expectation under the previous limit by
\[
\frac{\mc H_N(t)}{\gamma N} + \frac{1}{\gamma N}\log \bb E^N_{\rho_N(t,\cdot)}\bigg[\exp\Big(\gamma\sum_{x\in\T_N} \Big|\frac{1}{2\ell+1}\sum_{|y-x|\leqslant \ell} \varphi(\eta(y))-\overline\varphi\big(\rho\big(t,\tfrac x N\big)\big)\Big|\Big)\bigg].
\]
From Corollary \ref{cor:entropy}, the first term above vanishes as $N\to\infty$. As for the second term, we use the fact that $\nu_{\rho_N(t,\cdot)}^N$ is a product measure, 	and from H\"older's inequality we bound it from above by 
\begin{equation}\label{eq:riemansum}
\frac{1}{\gamma N}\sum_{x\in\T_N} \frac{1}{2\ell+1} \log \bb E^N_{\rho_N(t,\cdot)}\bigg[\exp\Big(\gamma \Big|\sum_{|y-x|\leqslant \ell} \varphi(\eta(y))-\overline{\varphi}\big(\rho\big(t,\tfrac x N\big)\big)\Big|\Big)\bigg].
\end{equation}
Since the profile $\rho(t,\cdot)$ is continuous on $\T$, and the function $\rho_N(t,\cdot)$ converges uniformly to $\rho(t,\cdot)$ (from Proposition \ref{prop:unif}) we deduce that \eqref{eq:riemansum} converges as $N\to\infty$ to
\[
\frac{1}{\gamma}\int_{\T} \frac{1}{2\ell+1}\log E_{\rho(t,u)}\bigg[\exp\Big(\gamma \Big|\sum_{|y|\leqslant \ell} \varphi(\eta(y))-\overline\varphi(\rho(t,u))\Big|\Big)\bigg] \; \d u,
\]
see also \eqref{eq:rieman-unif}. To conclude the proof, we proceed as in \cite[Chapter 6.1]{kl}: use the inequalities $e^x \leqslant 1+x+\frac12 x^2 e^{|x|}$ and $\log(1+x)\leqslant x$. Finally, choose $\gamma=\varepsilon / (2\ell+1)$. From the law of large numbers, last expression vanishes as $\ell \to \infty$ and then $\varepsilon \to 0$.
\end{proof}

\section{Norm bounds: statement and proof} \label{ssec:properties}

In this section we state and prove the bounds on the derivatives of the regularized solution that are needed for Proposition \ref{thm:entropy}. The latter will be  proved further in Section~\ref{sec:relative}.

\begin{proposition}\label{lem:Lip}
For any $N \in\N$, there holds
\begin{align}
  \sup_{(t,u)\in[0,T]\times\T}\big|\partial_u {\Pr}_N(t,u) \big| \leqslant &\; C_{\rm Lip}, \label{eq:p_Lip}\\
  \sup_{(t,u)\in[0,T]\times\T}\big|\partial_u \rho_N(t,u) \big| \leqslant &\;  \frac{C_{\rm Lip}}{m} \, (\eps_N)^{2-m},  \label{eq:rho_Lip}\\
  \iint_{[0,T]\times\T}|\p_{uu} {\Pr}_N(t,u)|^2\d t \d u \leq &\; \frac{(C_{\rm Lip})^2}{2m}\, (\eps_N)^{1-m}, \label{eq:L2H2_p}\\
  \iint_{[0,T]\times\T}|\p_{uu} \rho_N(t,u)|^2\d t \d u \leq &\; C_0(\eps_N)^{5-3m} 
,
  \label{eq:L2H2_rho}\
 \end{align}
where $C_{\rm Lip}$ has been defined in \eqref{eq:lip} and $C_0$ is related to $m$ and $C_{\rm Lip}$ as follows:
\[C_0=\frac{(C_{\rm Lip})^2}{m^3}
  \left( 1+ \frac{2(m-2)^2}{(3m-4)(3m-5)}\right).\]
\end{proposition}

\begin{proof}[Proof of Proposition \ref{lem:Lip}]
First, one can easily check that the space derivative of the pressure $f_N = \p_u {\Pr}_N$ satisfies
\be\label{eq:f_N}
\p_t f_N - \p_u\left(m(\rho_N)^{m-1} \p_u f_N + \left(f_N\right)^2 \right) = 0, \qquad \left(f_N\right)_{|_{t=0}} = \p_u {\Pr}_N^{\rm ini}. 
\ee
This equation has a maximum principle, so that $\|f_N\|_\infty \leq \|f_N^{\rm ini}\|_{\infty}$, which yields~\eqref{eq:p_Lip}
thanks to Assumption~\eqref{eq:lip}  and \eqref{eq:Lip-ini-reg}. A similar proof can be found in \cite[Prop. 15.4]{vaz}. Then it follows from~\eqref{eq:pressure-reg}  that 
\begin{equation}\label{eq:deriv}
\p_u \rho_N= \frac{\p_u {\Pr}_N}{m\,(\rho_N)^{m-2}},
\end{equation}
and estimate~\eqref{eq:rho_Lip} follows directly from~\eqref{eq:max-eps} and~\eqref{eq:p_Lip}.  
In order to get \eqref{eq:L2H2_p}, one multiplies~\eqref{eq:f_N} by $f_N$ and integrate over $[0,T]\times \T$, 
leading to 
\[
\frac12 \int_\T |f_N(T,u)|^2 \d u + \iint_{[0,T]\times\T} m\rho_N^{m-1} |\p_u f_N|^2 \d t \d u =
\frac12 \int_\T \left|\p_u {\Pr}_N^{\rm ini}\right|^2 \d u.
\]
Using \eqref{eq:max-eps} and~\eqref{eq:Lip-ini-reg} in the previous estimate (recalling that $\p_uf_N=\p_{uu}{\Pr}_N$) yields~\eqref{eq:L2H2_p}.
Let us finally establish~\eqref{eq:L2H2_rho}. To this end, remark first that $ {\Pr}_N = 2 \rho_N$ if $m=2$, so that 
\eqref{eq:L2H2_rho} directly follows from \eqref{eq:L2H2_p} in this case. Assume now that $m \geq 3$, then from \eqref{eq:pressure-reg} we get
\begin{align}
\p_{uu}\rho_N =& \frac1m \left( \rho_N^{2-m}\p_{uu}{\Pr}_N - (m-2) \rho_N^{1-m}\;\p_u \rho_N \; \p_u{\Pr}_N \right)\nonumber\\
= &  \frac1m \left( \rho_N^{2-m}\p_{uu}{\Pr}_N + \p_u \psi_1(\rho_N) \; \p_u{\Pr}_N \right) \label{eq:rho_uu}
\end{align}
where \cm{$\psi_1(\rho) = \rho^{2-m}$}.
Using $(a+b)^2 \le 2(a^2+b^2)$ and the previous estimates 
\eqref{eq:max-eps}, \eqref{eq:p_Lip}, and \eqref{eq:L2H2_p}, we obtain that 
\[
\big\|\p_{uu}\rho_N\big\|^2_{L^2((0,T)\times\T)}\leq \frac{2(C_{\rm Lip})^2}{m^2}\left(\frac{1}{2m}\left(\eps_N\right)^{5-3m}+
\big\|\p_u \psi_1(\rho_N)\big\|^2_{L^2((0,T)\times\T)}
\right).
\]
It remains to bound $\|\p_u\psi_1(\rho_N)\|^2_{L^2((0,T)\times\T)}$. To this end, 
multiply the PME~\eqref{eq:porous-smooth} by 
$$\varphi_1(\rho_N)= \frac{(m-2)^2}{m(4-3m)}\rho_N^{4-3m}$$ and integrate over $[0,T]\times\T$, leading to 
\be\label{eq:dpsiduL2_0}
\int_\T \Phi_1(\rho_N)(T,u)\d u + \big\|\p_u \psi_1(\rho_N)\big\|^2_{L^2((0,T)\times\T)} = \int_\T \Phi_1(\rho_N^{\rm ini})(u)\d u
\ee
with $\Phi_1(\rho) = \int_0^\rho \varphi_1(a) \d a = \frac{(m-2)^2}{m(4-3m)(5-3m)} \rho^{5-3m} \geq 0$ for $\rho>0$.
Then we deduce from~\eqref{eq:max-eps} that 
\be\label{eq:dpsiduL2}
 \big\|\p_u \psi_1(\rho_N)\big\|^2_{L^2((0,T)\times\T)}\leq \frac{(m-2)^2}{m(4-3m)(5-3m)}\left(\eps_N\right)^{5-3m}.
 \ee
 Estimate~\eqref{eq:L2H2_rho} directly follows from~\eqref{eq:dpsiduL2_0}--\eqref{eq:dpsiduL2}.

\end{proof}

\begin{proposition}\label{prop:xxx}
There exist three constants ${C_1}, C_2,C_3>0$ (that depend on $m$, $C_h$ and $C_{\rm Lip}$) such that, for any $N \in \N$,
\begin{align} 
\sup_{t\in[0,T]} \int_{\T} \big| \partial_{uu}\Pr_N(t,u)\big|^2 \d u &\leqslant C_1 (\eps_N)^{2-2m},
\label{eq:LinfH2_p}   \\
\iint_{[0,T]\times \T} \big|\partial_{uuu} \Pr_N(t,u)\big|^2 \; \d t\d u &\leqslant C_1 (\eps_N)^{3-3m} \label{eq:L2H3_p},\\
\sup_{t\in[0,T]} \int_{\T} \big| \partial_{uu}\rho_N(t,u)\big|^2 \d u &\leqslant C_2 (\eps_N)^{6-4m},
\label{eq:LinfH2_rho}   \\
\iint_{[0,T]\times \T} \big|\partial_{uuu} \rho_N(t,u)\big|^2 \; \d t\d u &\leqslant C_3 (\varepsilon_N)^{7-5m}. \label{eq:L2H3_rho}
\end{align}
\end{proposition}

\begin{proof}[Proof of Proposition \ref{prop:xxx}]
For any $N\in\N$, we set $g_N = \p_u f_N = \p_{uu}\Pr_N$. It is a smooth solution to the problem
\be\label{eq:g_n}
\p_t g_N - \p_u \left( m (\rho_N)^{m-1} \p_u g_N + (m+1) f_N g_N \right) = 0, \qquad 
(g_N)_{|_{t=0}} = \p_{uu} \Pr_N^{\rm ini}.
\ee
Multiplying~\eqref{eq:g_n} by $2 g_N$ and integrating over $[0,t^\star]\times \bbT $ for some arbitrary $t^\star \in [0,T]$ 
provides 
\begin{multline*}
\int_\bbT |g_N|^2(t^\star,u) \d u -  \int_\bbT \big|\p_{uu} \Pr_N^\text{ini}\big|^2 \d u \\
+  \iint_{[0,t^\star]\times \bbT} 2m (\rho_N)^{m-1} | \p_u g_N |^2 \;\d t\d u  
+ \iint_{[0,t^\star]\times \bbT} 2(m+1) f_N g_N \p_u g_N\; \d t\d u  = 0.
\end{multline*}
It follows from the inequality 
\[
2 |g_N \p_u g_N| \leq \frac{(2m-1) (\rho_N)^{m-1}}{(m+1)C_{\rm Lip}} \left|\p_u g_N\right|^2 
+
\frac{(m+1)C_{\rm Lip}}{(2m-1) (\rho_N)^{m-1}} \left|g_N\right|^2 
\]
and from estimate~\eqref{eq:p_Lip} that 
\begin{multline*}
 \iint_{[0,t^\star]\times \bbT} 2m (\rho_N)^{m-1} | \p_u g_N |^2 \d t\d u  
+ \iint_{[0,t^\star]\times \bbT} 2(m+1) f_N g_N \p_u g_N \d t\d u \\
\geq 
 \iint_{[0,t^\star]\times \bbT} (\rho_N)^{m-1} | \p_u g_N |^2 \d t\d u  
 -  \cm{\frac{(m+1)^2(C_{\rm Lip})^2}{2m-1}} \iint_{[0,t^\star]\times \bbT} \rho_N^{1-m} |g_N|^2 \d t \d u.
\end{multline*}
Therefore, we obtain that 
\begin{multline*}
\int_\bbT |g_N|^2(t^\star,u) \d u +  \iint_{[0,t^\star]\times \bbT} (\rho_N)^{m-1} | \p_u g_N |^2 \d t\d u
\\\leq 
 \int_\bbT \big|\p_{uu} \Pr_N^\text{ini}\big|^2 \d u
+ \cm{\frac{(m+1)^2(C_{\rm Lip})^2}{2m-1}} \iint_{[0,t^\star]\times \bbT} \rho_N^{1-m} |g_N|^2 \d t \d u.
\end{multline*}
We deduce from~\eqref{eq:max-eps} and~\eqref{eq:L2H2_p} that 
\[
 \iint_{[0,t^\star]\times \bbT} \rho_N^{1-m} |g_N|^2 \d t \d u \leq \frac{(C_{\rm Lip})^2}{2m} (\eps_N)^{2-2m}, 
\]
whereas the definition~\eqref{eq:rho-ini} of $\Pr_N^{\rm ini}$ ensures that 
\[
\left\|\p_{uu}\Pr_N^{\rm ini}\right\|_\infty \leq C_{\rm Lip} \big\|\p_y h_N\big\|_{1} = \frac{2 C_h C_{\rm Lip}}{\eps_N},
\]
where the last equality follows from \eqref{eq:h_n}.
Since $\eps_N \leq \frac12$, and since $\rho_N \geq \eps_N$, we obtain that~\eqref{eq:LinfH2_p} and \eqref{eq:L2H3_p} 
hold for
\[
C_1 = (C_{\rm Lip})^2\left(\left(C_h\right)^2 2^{2m-3} +  \cm{\frac{(m+1)^2}{2m(2m-1)}(C_{\rm Lip})^2}\right).
\]
Using \eqref{eq:deriv}, formula~\eqref{eq:rho_uu} can be recast into 
\[
\p_{uu}\rho_N = \frac1m \left( \rho_N^{2-m}\p_{uu}{\Pr}_N - \cm{\frac{m-2}{m}} \;\rho_N^{3-2m}\;\left(\p_u{\Pr}_N\right)^2 \right).
\]
Therefore, using $(a+b)^2 \leq 2(a^2 + b^2)$ again as well as estimates~\eqref{eq:max-eps}, 
\eqref{eq:p_Lip} and \eqref{eq:LinfH2_p}, we obtain that 
\[
\sup_{t \in [0,T]} \int_\T |\p_{uu}\rho_N(t,u)|^2 \d u \leq \frac{2 C_1}{m^2} (\eps_N)^{\cm{6-4m}} + 
2\left(\frac{m-2}{\cm{m^2}}\right)^2\left(C_{\rm Lip}\right)^2(\eps_N)^{6-4m}.
\]
Estimate~\eqref{eq:LinfH2_rho} follows from the above inequality, the constant $C_2$ being 
given by 
\[
C_2 =  \frac{2 C_1}{\cm{m^2}} + 2\left(\frac{m-2}{\cm{m^2}}\right)^2\left(C_{\rm Lip}\right)^2.
\]
Finally, from \eqref{eq:rho_uu}, the third derivative of $\rho_N$ is the sum of four terms: 
\begin{align}
\partial_{uuu}\rho_N = & \frac{1}{m} \; \rho_N^{2-m}\; \partial_{uuu}\Pr_N \label{eq:term11}\\
& + \frac{2(2-m)}{m}\; \rho_N^{1-m} \; \partial_u\rho_N\; \partial_{uu}\Pr_N \label{eq:term12}\\
& + \frac{(2-m)}{m}\; \rho_N^{1-m}\; \partial_{uu}\rho_N \; \partial_u \Pr_N \label{eq:term13}\\
& + \frac{(2-m)}{m}\; \partial_u\psi_2(\rho_N)\; \partial_u\rho_N\; \partial_u\Pr_N, \label{eq:term14}
\end{align}
where $\psi_2(\rho):=\rho^{1-m}$. Then, in order to bound the integral $\iint_{[0,T]\times \T} |\p_{uuu}\rho_N(t,u)|^2\; \d t\d u$, we use the inequality $(a_1+a_2+a_3+a_4)^2 \leq 4(a_1^2+a_2^2+a_3^2+a_4^2)$, and we compute the contribution of each term \eqref{eq:term11}--\eqref{eq:term14} using the previous estimates, as follows: from \eqref{eq:max-eps} and \eqref{eq:L2H3_p}, we have
\[\iint_{[0,T]\times \T} \left(\frac{1}{m} \; \rho_N^{2-m}\; \partial_{uuu}\Pr_N\right)^2 \d t\d u\leq \frac{C_1}{m^2}(\varepsilon_N)^{7-5m}.\]
Second, from \eqref{eq:max-eps}, \eqref{eq:rho_Lip} and \eqref{eq:L2H2_p}, we get
\[\iint_{[0,T]\times \T} \left(\frac{2(2-m)}{m}\; \rho_N^{1-m} \; \partial_u\rho_N\; \partial_{uu}\Pr_N\right)^2\d t\d u\leq \frac{2(2-m)^2(C_{\rm Lip})^{4}}{m^5}(\varepsilon_N)^{7-5m}.\]
In the same way, from \eqref{eq:max-eps}, \eqref{eq:L2H2_rho} and \eqref{eq:p_Lip},
\[\iint_{[0,T]\times \T} \left(\frac{(2-m)}{m}\; \rho_N^{1-m}\; \partial_{uu}\rho_N \; \partial_u \Pr_N\right)^2 \d t\d u\leq \frac{(2-m)^2\;C_0(C_{\rm Lip})^2}{m^2}(\varepsilon_N)^{7-5m}.\]
It remains to estimate the contribution of \eqref{eq:term14}. For this term, we use the same strategy as in the end of the proof of \eqref{eq:L2H2_rho}. First, we bound it from \eqref{eq:p_Lip} and \eqref{eq:rho_Lip} as follows:
\begin{multline*}\iint_{[0,T]\times \T} \left(\frac{(2-m)}{m}\; \partial_u\psi_2(\rho_N)\; \partial_u\rho_N\; \partial_u\Pr_N\right)^2 \d t\d u\\\leq \frac{(2-m)^2(C_{\rm Lip})^4}{m^{4}}(\varepsilon_N)^{4-2m} \big\|\partial_u\psi_2(\rho_N)\big\|^2_{L^2((0,T)\times\T)} .\end{multline*}
Finally, to estimate the $L^2$ norm $\|\p_u\psi_2(\rho_N) \|_{L^2((0,T)\times\T)}^2$, we multiply the PME \eqref{eq:porous-smooth} by 
\[\varphi_2(\rho_N) = \frac{(1-m)^2}{m(2-3m)}\rho_N^{2-3m}\] and integrate over $[0,T]\times\T$. This easily leads to 
\[\big\|\partial_u\psi_2(\rho_N)\big\|^2_{L^2((0,T)\times\T)}\leq \frac{(1-m)^2}{m(2-3m)(3-3m)}(\varepsilon_N)^{3-3m}. \] 
Finally, collecting all the contributions coming from \eqref{eq:term11}--\eqref{eq:term14} (which all are of the same order), we obtain the bound \eqref{eq:L2H3_rho} with \[C_3= \frac{4}{m^2}\left( C_1 + \frac{2(2-m)^2(C_{\rm Lip})^2}{m^3} + (2-m)^2C_0(C_{\rm Lip})^2 + \frac{(2-m)^2(1-m)^2(C_{\rm Lip})^4}{m(2-3m)(3-3m)} \right).\]
\end{proof}

We conclude this section by getting some technical bounds on the norms of $\lambda_N$ and its derivatives, where $\lambda_N$ has been defined in function of $\rho_N$ in \eqref{eq:lambdaN}.

\begin{proposition}\label{prop:lambdaN}
For any $N \in \bb N$, 
\begin{equation}\sup_{(t,u) \in [0,T]\times \T} \big|\partial_u \lambda_N(t,u)\big| \leqslant \cm{\frac{2C_{\rm Lip}}{m} (\varepsilon_N)^{1-m}}.\label{eq:lambdabound1}\end{equation}
 Moreover, \cm{there exists $C>0$ which depends on $m,C_h$ and $C_{\rm Lip}$} such that
\begin{equation}
\label{eq:lambdabound2}
\sup_{t\in[0,T]}\int_{\T} \big|\partial_{uu}\lambda_N\big|^2 (t,u) \; \d u \leqslant \cm{C(\varepsilon_N)^{4-4m}}\end{equation}
and finally 
\begin{align}
\label{eq:lambdabound3}
\iint_{[0,T]\times\T} \big|\partial_{uuu} \lambda_N\big|^2(t,u) \; \d t \d u &\leqslant \cm{C(\varepsilon_N)^{6-6m}}\\
\iint_{[0,T]\times\T} \big|\partial_{u}\partial_t \lambda_N\big|^2(t,u) \; \d t \d u &\leqslant \cm{C(\varepsilon_N)^{6-6m}}. \label{eq:lambdabound4}
\end{align}
\end{proposition}

\begin{proof}[Proof of Proposition \ref{prop:lambdaN}]\cm{Using only the definition \eqref{eq:lambdaN}, one can easily prove the following:
\begin{lemma} \label{prop:derivatives}
We have
\begin{align}
\partial_u \lambda_N & = \frac{\p_u\rho_N}{\rho_N}+\frac{\p_u\rho_N}{1-\rho_N}= \frac{\partial_u \rho_N}{\rho_N(1-\rho_N)} \vphantom{\Bigg[} \label{eq:diff1}\\
\partial_{uu} \lambda_N &= (\partial_{uu}\rho_N)\bigg(\frac{1}\rho_N + \frac{1}{1-\rho_N}\bigg)+ (\partial_u\rho_N)^2\bigg(\frac{1}{(1-\rho_N)^2} - \frac{1}{\rho_N^2}  \bigg) \vphantom{\Bigg[} \label{eq:diff2}\\
\partial_{uuu}\lambda_N &= (\partial_{uuu}\rho_N)\bigg(\frac{1}\rho_N + \frac{1}{1-\rho_N}\bigg) + 3 (\partial_{uu} \rho_N)( \partial_u \rho_N)\bigg(\frac{1}{(1-\rho_N)^2} - \frac{1}{\rho_N^2}  \bigg) \\
& \quad +2 (\partial_u\rho_N)^3 \bigg(\frac{1}{\rho_N^3}+\frac{1}{(1-\rho_N)^3}\bigg). \vphantom{\Bigg[} \label{eq:diff3}
\end{align}
Therefore, if $\rho_N$ is solution to the porous medium equation $\partial_t \rho_N = \partial_{uu}(\rho_N^m)$ then 
\begin{equation} \partial_t \lambda_N = m\rho_N^{m-1}\; \partial_{uu} \lambda_N + m\rho_N^{m-1}(m-(m+1)\rho_N)\; (\partial_u \lambda_N)^2. \label{eq:identity1}\end{equation}
\end{lemma}
Then, from \eqref{eq:diff1} and \eqref{eq:max-eps}, we have
\[ \|\p_u\lambda_N\|_\infty \leq   \frac{2\|\p_u\rho_N\|_\infty}{\varepsilon_N}.\] Therefore, the first bound \eqref{eq:lambdabound1} is straightforward from Proposition \ref{lem:Lip}. In the same way, using Lemma \ref{prop:derivatives} together with \eqref{eq:max-eps} and the inequality $(a+b)^2 \leq 2(a^2 + b^2)$, we get, for any $t \in [0,T]$, that 
\[
\big\|\partial_{uu}\lambda_N(t,\cdot)\big\|_{L^2(\T)}^2 \leqslant 8\Bigg( \frac{\big\|\partial_{uu}\rho_N(t,\cdot)\big\|_{L^2(\T)}^2}{(\varepsilon_N)^2} + \frac{\big\|\partial_u\rho_N\big\|_{\infty}^4}{(\varepsilon_N)^4}\Bigg),
\]
therefore \eqref{eq:lambdabound2} follows from Proposition \ref{lem:Lip} and Proposition \ref{prop:xxx}, with $C$ that satisfies $C \geq 8(C_2+(C_{\rm Lip})^4/m^4)$. 
Also, from Lemma \ref{prop:derivatives}, from \eqref{eq:max-eps} and the inequality $(a+b+c)^2 \leq 3(a^2 + b^2+c^2)$, we have\[ \big|\partial_{uuu}\lambda_N\big|^2\leqslant 9\frac{\big|\partial_{uuu}\;\rho_N\big|^2}{(\varepsilon_N)^2} +36 \frac{\big|\partial_{uu}\rho_N\big|^2 \; \big|\partial_u \rho_N\big|^2}{(\varepsilon_N)^4} +16 \frac{\big|\partial_u\rho_N\big|^6}{(\varepsilon_N)^6}.\]
Then from Proposition \ref{lem:Lip} and Proposition \ref{prop:xxx}, we  easily obtain
\begin{align*} \iint_{[0,T]\times\T} \big|\partial_{uuu} \lambda_N\big|^2(t,u) \; \d t \d u & \leq \left(9C_3 + \frac{36(C_{\rm Lip})^2C_0}{m^2}\right)(\varepsilon_N)^{5-5m}+\frac{16(C_{\rm Lip})^6}{m^6}(\varepsilon_N)^{6-6m} \\
&\leq C (\varepsilon_N)^{6-6m},\end{align*}
with $C$ that satisfies 
\[ C \geq  \frac{1}{2^{m-1}}\left(9C_3 + \frac{36 (C_{\rm Lip})^2C_0}{m^2}\right) + \frac{16(C_{\rm Lip})^6}{m^6},\]
so that \eqref{eq:lambdabound3} is proved. Finally, to get \eqref{eq:lambdabound4}, we use \eqref{eq:identity1} together with $|\rho_N|\leq 1$, and we obtain that there exists a constant $\kappa=\kappa(m)$ which depends only on $m$ such that
\[ \big|\partial_u\partial_t\lambda_N\big|^2 \leqslant 
\kappa(m)\Big(\big|\partial_u\rho_N\big|^2 \big|\partial_{uu}\lambda_N\big|^2 + \big|\partial_{uuu}\lambda_N\big|^2 +  \big|\partial_u\lambda_N\big|^4 \big|\partial_u\rho_N\big|^2+ \big|\partial_u\lambda_N\big|^2 \big|\partial_{uu}\lambda_N\big|^2\Big), \]
and therefore, we let the reader conclude from Proposition \ref{lem:Lip} and the  three first bounds \eqref{eq:lambdabound1}, \eqref{eq:lambdabound2} and \eqref{eq:lambdabound3}, in order to get \eqref{eq:lambdabound4}.} 
\end{proof}

\section{Relative entropy estimates}\label{sec:relative}

In this section we prove Lemma \ref{l:smallentropy} and Proposition \ref{thm:entropy}. 

\subsection{Proof of Lemma~\ref{l:smallentropy}}\label{s:smallentropy}

We say that a configuration $\eta \in \{0,1\}^{\T_N}$ is $\rho$-\emph{compatible} with a profile $\rho:\T\to [0,1]$ if
\[ \eta(x) = \rho\big(\tfrac x N\big) \quad \text{ whenever } \rho\big(\tfrac x N\big)=0 \text{ or } 1. \]
Recall Definition \ref{eq:nurhoN}. Since $\rho_N^{\rm ini}\in [\varepsilon_N,1-\varepsilon_N]$, we can easily compute
\begin{align*}
\mathcal{H}_N(0)=&\sum_{\eta\ \rho^{\mathrm{ini}}\text{--comp.}}\nu_{\rho^{\rm ini}}^N(\eta)\Bigg\{\sum_{x\colon\rho^{\rm ini}(\frac x N)=0}\log\frac{1}{1-\rho^{\rm ini}_N(\frac x N)}+\sum_{x\colon\rho^{\rm ini}(\frac x N)=1}\log\frac{1}{\rho^{\rm ini}_N(\frac x N)}\\
&+\sum_{x\colon\rho^{\rm ini}(\frac x N)\in(0,1)}\left(\eta(x)\log\frac{\rho^{\rm ini}(\frac x N)}{\rho^{\rm ini}_N( \frac x N)}+(1-\eta(x))\log\frac{1-\rho^{\rm ini}(\frac x N)}{1-\rho^{\rm ini}_N(\frac x N)}\right)\Bigg\},
\end{align*}
where the first sum is over configurations $\eta\in\{0,1\}^{\T_N}$ compatible with the density profile $\rho^{\rm ini}$. Then, 
\begin{align}
\mathcal{H}_N(0)=&\sum_{x\colon\rho^{\rm ini}(\frac x N)=0}\log\frac{1}{1-\rho^{\rm ini}_N(\frac x N)}+\sum_{x\colon\rho^{\rm ini}(\frac x N)=1}\log\frac{1}{\rho^{\rm ini}_N(\frac x N)}\label{eq:smallent1}\\
&+\sum_{x\colon\rho^{\rm ini}(\frac x N)\in(0,1)}\left(\rho^{\rm ini}\big(\tfrac x N\big)\log\frac{\rho^{\rm ini}(\frac x N)}{\rho^{\rm ini}_N(\frac x N)}+\big(1-\rho^{\rm ini}\big(\tfrac x N\big)\big)\log\frac{1-\rho^{\rm ini}(\frac x N)}{1-\rho^{\rm ini}_N(\frac x N)}\right)\label{eq:smallent2}
\end{align}
The lemma then follows from \eqref{eq:conv}: indeed, there exists $C>0$ such that for all $x\in\T_N$, 
\begin{eqnarray}
\rho^{\rm ini}\big(\tfrac x N\big)=0&\Longrightarrow&\bigg|\log\frac{1}{1-\rho^{\rm ini}_N\big(\tfrac x N\big)}\bigg|\leq C \cm{(\varepsilon_N)^{\frac{1}{m-1}}},\\
\rho^{\rm ini}\big(\tfrac x N\big)=1&\Longrightarrow&\bigg|\log\frac{1}{\rho^{\rm ini}_N\big( \tfrac x N\big)}\bigg|\leq C\cm{(\varepsilon_N)^{\frac{1}{m-1}}}.
\end{eqnarray}
Therefore, we can bound \eqref{eq:smallent1} by $CN\cm{(\varepsilon_N)^{\frac{1}{m-1}}}$. In order to bound the first term in \eqref{eq:smallent2}, note that (using again \eqref{eq:conv}) there exists $C>0$ such that 
\begin{itemize}
\item if $\rho^{\rm ini}(\frac x N)\leq 2\cm{C_{\rm ini}(\varepsilon_N)^{\frac{1}{m-1}}}$, then \[\rho^{\rm ini}\big(\tfrac x N\big)\log\frac{\rho^{\rm ini}(\frac x N)}{\rho^{\rm ini}_N(\frac x N)}\leq C\cm{(\varepsilon_N)^{\frac{1}{m-1}}}|\log\varepsilon_N|,\]
\item if $\rho^{\rm ini}(\frac x N)> 2\cm{C_{\rm ini}(\varepsilon_N)^{\frac{1}{m-1}}}$, then \[\left|\frac{\rho^{\rm ini}(\frac x N)-\rho^{\rm ini}_N(\frac x N)}{\rho^{\rm ini}(\frac x N)}\right|<\frac12\] and
\begin{align*}
\left|\rho^{\rm ini}\big(\tfrac x N\big)\log\frac{\rho^{\rm ini}(\frac x N)}{\rho^{\rm ini}_N(\frac x N)}\right|&=\left|\rho^{\rm ini}\big(\tfrac x N\big)\log\bigg(1-\frac{\rho^{\rm ini}(\frac x N)-\rho^{\rm ini}_N(\frac x N)}{\rho^{\rm ini}(\frac x N)}\bigg)\right|\\
&\leq C\left|\rho^{\rm ini}\big(\tfrac x N\big)-\rho^{\rm ini}_N\big( \tfrac x N\big)\right|\leq C\cm{C_{\rm ini}(\varepsilon_N)^{\frac{1}{m-1}}}. \vphantom{\Bigg(}
\end{align*}
\end{itemize}
The second term in \eqref{eq:smallent2} is bounded similarly. Lemma~\ref{l:smallentropy} follows.

\bigskip

We now turn to the proof of Proposition \ref{thm:entropy}, which is the central result of this work.

\subsection{Entropy production}
\label{sec:proof}

First of all, the following well-known entropy estimate is due to Yau \cite{yau}:
\[
\partial_t \mc H_N(t) \leqslant \int \bigg\{ \frac{N^2 \mathcal{L}_N \psi_t^N}{\psi_t^N}-\partial_t \log(\psi_t^N) \bigg\} \d\mu_{t}^N.
\]
Let us denote \begin{align*}
h(\eta)&:
=\sum_{y=-m+1}^0\prod_{z=y}^{y+m-1}\eta(z)-\sum_{y=-m+1}^{-1}\prod_{\substack{z=y\\z\neq 0}}^{y+m}\eta(z),\\
g(\eta)&:=\frac{1}{2}\cm{r_{0,1}(\eta)}(\eta(0)-\eta(1))^2.
\end{align*}
Note that $\overline g(\rho)=m\rho^m(1-\rho)$ and $\overline h(\rho)=\rho^m$, and also $|g(\eta)|\leqslant m$ and $|h(\eta)|\leqslant 2m$ for any $\eta$. 
We first prove the following technical result:
\begin{lemma} \label{lem:entropy} Under Assumption \ref{ass:eps}, namely assuming $(\varepsilon_N)^{\cm{6m-6}}N\to \infty$, we have
\begin{align}
& \int \bigg\{\frac{N^2 \mathcal{L}_N \psi_t^N}{\psi_t^N} -\partial_t \log(\psi_t^N)\bigg\}\d\mu_{t}^N \notag \\
&\; = \int \sum_{x \in \T_N} \partial_{uu} \lambda_N\big(t,\tfrac{x}{N}\big) \bigg\{\tau_x h(\eta)-\overline h\big({\rho}_N\big(t,\tfrac{x}{N}\big)\big)- \overline h'\big({\rho}_N\big(t,\tfrac{x}{N}\big)\big)\Big(\eta(x)-{\rho}_N\big(t,\tfrac{x}{N}\big) \Big)\bigg\} \d\mu_{t}^N \label{eq:taylor1}\\
&  \; + \int \sum_{x \in \T_N} (\partial_u \lambda_N)^2\big(t,\tfrac{x}{N}\big) \bigg\{\tau_x g(\eta)-\overline g\big({\rho}_N\big(t,\tfrac{x}{N}\big)\big)- \overline g'\big({\rho}_N\big(t,\tfrac{x}{N}\big)\big)\Big(\eta(x)-{\rho}_N\big(t,\tfrac{x}{N}\big) \Big)\bigg\}\d\mu_{t}^N \label{eq:taylor2}\\
& \; + \delta(t,N) \vphantom{\bigg[}, \notag
\end{align}
where
 \[\frac1N\bigg|\int_0^T\delta(t,N)\d t\bigg|\xrightarrow[N\to\infty]{}0.\]
\end{lemma}

\begin{proof}[Proof of Lemma \ref{lem:entropy}] Fix $t\in[0,T]$. For the sake of brevity we denote $\lambda_x^N:=\lambda_N(t,\frac x N)$. 
\medskip

\paragraph{\sc Step 1 -- Part coming from the generator: } 
First we have
\begin{align}
\frac{N^2 \mc L_N\psi_t^N}{\psi_t^N} = & N^2 \sum_{x\in\T_N}r_{x,x+1}(\eta)\eta(x)\big(1-\eta(x+1)\big)
\Big(e^{\lambda^N_{x+1}-\lambda^N_x}-1\Big) \label{eq:lpsi1}\\
& + N^2 \sum_{x\in\T_N}r_{x,x+1}(\eta)\eta(x+1)\big(1-\eta(x)\big)
\Big(e^{\lambda^N_{x}-\lambda^N_{x+1}}-1\Big).\label{eq:lpsi2}\end{align}
In \eqref{eq:lpsi1} and \eqref{eq:lpsi2} we write the exponential as the infinite sum: $e^z - 1 = \sum_{k\geqslant 1} \frac{z^k}{k!}$. The first order term $(k=1)$ gives:
\begin{align*}  N^2 \sum_{x\in\T_N} &r_{x,x+1}(\eta) \big(\eta(x)-\eta(x+1)\big) \big(\lambda_{x+1}^N - \lambda_x^N\big)\\
&=\; N^2\sum_{x\in\T_N} \Bigg(\sum_{y=x-m+1}^x\prod_{\substack{z=y\\z\neq x+1}}^{y+m}\eta(z)- \sum_{y=x-m+1}^x\prod_{\substack{z=y\\z\neq x}}^{y+m}\eta(z)\Bigg) \big(\lambda_{x+1}^N  - \lambda_x^N \big)\\
&=\; N^2\sum_{x\in\T_N} \big(\tau_xh(\eta)-\tau_{x+1}h(\eta)\big)\big(\lambda_{x+1}^N  - \lambda_x^N \big) \vphantom{\Bigg(}\\
&=\;  N^2 \sum_{x\in\T_N} \tau_{x}h(\eta) \big(\lambda_{x+1}^N + \lambda_{x-1}^N - 2 \lambda_x^N \big).\vphantom{\Bigg(}
\end{align*} In order to replace the discrete Laplacian by its continuous version, let us estimate the following error
\begin{align*}
r_N(t)&:=\bigg| \int \sum_{x\in\T_N} \tau_{x}h(\eta)\; \Big( N^2 \big(\lambda_{x+1}^N + \lambda_{x-1}^N - 2 \lambda_x^N \big) - \partial_{uu}\lambda_N\big(t,\tfrac x N\big)\Big)\; \d\mu_{t}^N\bigg|.
\\
& \leqslant 2m  \sum_{x\in\T_N}\bigg|\Big( N^2 \big(\lambda_{x+1}^N + \lambda_{x-1}^N - 2 \lambda_x^N \big) - \partial_{uu}\lambda_N\big(t,\tfrac x N\big)\Big)\bigg|,\end{align*} where the last inequality comes from the fact $|h(\eta)|\leqslant \cm{2m}$. 
We use the Taylor formula for the smooth function $u \mapsto \lambda_N(t,u)$ in order to obtain
\begin{multline}
N^2\big(\lambda_{x+1}^N+\lambda_{x-1}^N - 2\lambda_x^N\big) - \partial_{uu}\lambda_N\big(t,\tfrac x N\big) \\ = \frac{N^2}2\int_{\frac x N}^{\frac{x+1}N} \partial_{uuu}\lambda_N(t,u)\big(\tfrac{x+1}{N}-u\big)^2\; \d u -  \frac{N^2}2\int_{\frac{x -1}N}^{\frac{x}N} \partial_{uuu}\lambda_N(t,u)\big(\tfrac{x-1}{N}-u\big)^2\; \d u .  \label{eq:taylor}
\end{multline}
We start with the first integral in \eqref{eq:taylor}. The second one is very similar and the same argument will work. We use several times the Cauchy-Schwarz inequality in order to write 
\begin{align} 
N^2\sum_{x\in\T_N}\bigg| \int_{\frac x N}^{\frac{x+1}N} &\partial_{uuu}\lambda_N(t,u)\big(\tfrac{x+1}{N}-u\big)^2\; \d u\bigg| \notag \\ 
& \leqslant  N^2   \sum_{x\in\T_N} \bigg\{\bigg(\int_{\frac{x}{N}}^{\frac{x+1}{N}} \big|\partial_{uuu}\lambda_N\big|^2(t,u)\; \d u\bigg)^{\frac12} \bigg(\int_{\frac x N}^{\frac{x+1}{N}} \big(\tfrac{x+1}{N}-u\big)^4\; \d u\bigg)^{\frac12}\bigg\} \notag \\
&\leqslant \frac{N^2}{\sqrt{5}N^{\frac52}} \sum_{x\in\T_N} \bigg(\int_{\frac{x}{N}}^{\frac{x+1}{N}} \big|\partial_{uuu}\lambda_N\big|^2(t,u)\; \d u\bigg)^{\frac12}
\notag \\ 
& \leqslant  \frac{N^2}{\sqrt{5}N^{\frac52}} \sqrt{N}\; \bigg\{ \sum_{x\in\T_N} 
\int_{\frac{x}{N}}^{\frac{x+1}{N}} \big|\partial_{uuu}\lambda_N\big|^2(t,u)\; \d u\bigg\}^{\frac12} \notag \\
& = \;\frac{1}{\sqrt{5}} \big\|\partial_{uuu}\lambda_N(t,\cdot)\big\|_2. \label{eq:int3}
\end{align}
Recall Proposition \ref{prop:lambdaN}: we have proved that
\[\iint_{[0,T]\times \T} \big|\partial_{uuu}\lambda_N\big|^2(t,u) \; \d t\d u \leqslant \cm{C (\varepsilon_N)^{6-6m}},\] for some $C>0$. We let the reader repeat the argument for the second integral in \eqref{eq:taylor}, and deduce the following:
\begin{equation}
\int_0^T r_N(t)\; \d t \leqslant C'\sqrt{T} \cm{(\varepsilon_N)^{3-3m}}, \label{eq:error1}
\end{equation} for some $C'>0$. 
From Assumption \ref{ass:eps}, we get \cm{$N(\varepsilon_N)^{3m-3}=(N (\varepsilon_N)^{6m-6})^\frac12 \sqrt N \to\infty$}, and we then have \[\frac1N\int_0^T r_N(t) \d t \xrightarrow[N\to\infty]{} 0.\] Therefore, the first order term ($k=1$) gives the first contribution in \eqref{eq:taylor1}, namely 
\[\int \sum_{x\in\T_N} \partial_{uu}\lambda_N\big(t,\tfrac x N\big)\tau_xh(\eta) \; \d \mu_t^N\] plus an error $r_N(t)$ that we include in $\delta(t,N)$.

In the same way, the second order term ($k=2$) gives
\begin{multline}
N^2 \sum_{x\in\T_N} \frac12 r_{x,x+1}(\eta) \big(\eta(x)-2\eta(x)\eta(x+1)+\eta(x+1)\big) \big(\lambda_{x+1}^N - \lambda_x^N\big)^2 \\
= N^2 \sum_{x\in\T_N} \tau_x g(\eta) \big(\lambda_{x+1}^N - \lambda_x^N\big)^2.
\end{multline}
We want here to estimate the error
\[
s_N(t):= \bigg| \int \sum_{x\in\T_N}\tau_xg(\eta) \Big(N^2\big(\lambda_{x+1}^N-\lambda_x^N\big)^2 - (\partial_u\lambda_N)^2\big(t,\tfrac x N\big)\Big) \d \mu_t^N\bigg|.
\] As before, the Taylor formula and the Cauchy-Schwarz inequality allows us to bound
\begin{align*}
s_N(t) & \leqslant 2mN \sum_{x\in\T_N}\bigg|  \partial_u\lambda_N \big(t,\tfrac x N\big) \int_{\frac x N}^{\frac{x+1} N} \partial_{uu}\lambda_N(t,u)\big(\tfrac{x+1}{N}-u\big)\; \d u  \bigg|  \\
& \qquad +m N^2  \sum_{x\in\T_N}\bigg| \int_{\frac x N}^{\frac{x+1} N} \partial_{uu}\lambda_N(t,u)\big(\tfrac{x+1}{N}-u\big)\; \d u\bigg|^2 \\
& \leqslant  \frac{2mN}{\sqrt{3}N^{\frac32}} \; \big\|\partial_u\lambda_N(t,\cdot)\big\|_\infty \sum_{x\in\T_N}\bigg|\int_{\frac x N}^{\frac{x+1}{N}} \big|\partial_{uu}\lambda_N(t,u)\big|^2 \; \d u \bigg|^{\frac12} \\
& \qquad + \frac{mN^2}{3N^3} \sum_{x\in\T_N} \int_{\frac x N}^{\frac{x+1}{N}} \big|\partial_{uu}\lambda_N(t,u)\big|^2 \; \d u  \\
& \leqslant \frac{2m}{\sqrt{3}} \big\| \partial_u\lambda_N(t,\cdot)\big\|_\infty \; \big\|\partial_{uu}\lambda_N(t,\cdot)\big\|_2 + \frac{m}{3N} \; \big\|\partial_{uu}\lambda_N(t,\cdot)\big\|_2^2 \vphantom{\bigg\{}\\
& \leqslant \cm{C''\bigg( (\varepsilon_N)^{3-3m} + \frac{(\varepsilon_N)^{4-4m}}{N}\bigg)}
\end{align*}
for some $C''>0$, where the last inequality follows from Proposition \ref{prop:lambdaN}. Therefore, we also get that $\frac1N \int_0^T s_N(t)\; \d t \to 0$, and the second order term gives the first contribution in \eqref{eq:taylor2}, namely 
\[ \int \sum_{x\in\T_N} (\partial_u \lambda_N)^2\big(t,\tfrac x N\big) \tau_x g(\eta)\; \d\mu_t^N,\] plus that error $s_N(t)$ that we include in $\delta(t,N)$.

 Finally, we show that none of the higher order terms ($k \geq 3$) contributes and they are all included in $\delta(t,N)$. Precisely, we estimate  
 \begin{equation} N^2\int_0^T \frac1N \sum_{x\in\T_N} \sum_{k \geqslant 3} \frac{|\lambda_{x+1}^N-\lambda_x^N|^k}{k!}\; \d t  \label{eq:higher}\end{equation} and show that this quantity vanishes as $N\to\infty$. Using Proposition \ref{prop:lambdaN}, we bound \eqref{eq:higher} from above by 
 \begin{align*}
 N^2\int_0^T \sum_{k\geqslant 3} \frac{\|\partial_u\lambda_N(t,\cdot)\|_\infty^k}{k!\; N^k} \; \d t & \leqslant TN^2 \sum_{k \geqslant 3} \frac{C^k \; \cm{(\varepsilon_N)^{k(1-m)}}}{k!\; N^k} \\
 & = TN^2\Big(e^{\cm{\frac C N (\varepsilon_N)^{1-m}}} - \frac{C^2\cm{(\varepsilon_N)^{2(1-m)}}}{2N^2}-\frac{C\cm{(\varepsilon_N)^{1-m}}}{N}-1\Big).
 \end{align*}
 with $C=\cm{2C_{\rm Lip}/m}$. For any $x \in [0,1]$ we have $e^x-\frac{x^2}{2}-x-1 \leqslant x^3$, therefore the last expression above is bounded by 
 \[ TC^3 \; \frac{N^2 \cm{(\varepsilon_N)^{3-3m}}}{N^3} = \frac{TC^3\cm{(\varepsilon_N)^{3-3m}}}{N}  \xrightarrow[N\to\infty]{}0, \] 
 from  Assumption \ref{ass:eps}.

\medskip

\paragraph{\sc Step 2 -- Part coming from  $\log(\psi_t^N)$: } The term with $\log(\psi_t^N)$ can be explicitly computed as
\begin{align*}\partial_t\log(\psi_t^N) &=  \sum_{x\in\T_N} \partial_t\lambda_N\big(t,\tfrac x N\big) \Big[\eta(x)-\int \eta(x)\psi_t^N(\eta)d\nu_\alpha(\eta)\Big] \\ 
& =  \sum_{x\in\T_N} \partial_t\lambda_N\big(t,\tfrac x N\big) \Big[\eta(x)-\rho_N\big(t,\tfrac x N\big)\Big].\end{align*}
Recall that by definition $\overline{g}(\rho)=m\rho^m(1-\rho)$ and $\overline{h}(\rho)=\rho^m$. A straightforward computation (see Lemma \ref{prop:derivatives}) gives
\begin{equation}\label{eq:ident}\partial_t\lambda_N = \partial_{uu}\lambda_N \; \overline{h}'(\rho_N) + (\partial_u\lambda_N)^2\; \overline{g}'(\rho_N).\end{equation}
Therefore, this term appears exactly on that form in  \eqref{eq:taylor1} and \eqref{eq:taylor2}.

\medskip

\paragraph{\sc Step 3 -- Additional term: } Note that in \eqref{eq:taylor1} and \eqref{eq:taylor2} there is an extra term, that does not appear from the previous computations. Therefore, we have to substract it, and use the triangular inequality to estimate it. We show that that term  is actually of order $o(N)$ when integrated  in time between $0$ and $T$, and  therefore goes in $\delta(t,N)$. Indeed, the extra term  reads
\[\sum_{x\in\T_N} F_N\big(t,\tfrac x N\big)\] where
\[ F_N(t,u):= \partial_{uu}\lambda_N(t,u)\:\overline h(\rho_N(t,u))+(\partial_u\lambda_N)^2(t,u)\;\overline g(\rho_N(t,u)).\]
We want to show that \begin{equation}
\label{eq:obj}
\frac1N \bigg|\int_0^T \sum_{x\in\T_N} F_N\big(t,\tfrac x N\big) \; \d t\bigg| \xrightarrow[N\to\infty]{} 0.
\end{equation}
First, note that, for any $t >0$,
\[\int_{\T} F_N(t,u)\; \d u 
 = \int_\T \partial_u \bigg(\frac{\cm{(\rho_N)^{m-1}} \; \partial_u\rho_N}{1-\rho_N}\bigg)(t,u) \; \d u = 0.
\]
Therefore, to prove \eqref{eq:obj} it is enough to prove that the following quantity vanishes:
\[
\int_0^T\Big|\frac1N\sum_{x\in\T_N} F_N\big(t,\tfrac x N\big) - \int_{\T} F_N(t,u)\; \d u \Big| \d t \leqslant 
\int_0^T \sum_{x\in \T_N} \int_{\frac{x}{N}}^{\frac{x+1}{N}} \Big|F_N\big(t,\tfrac{x}{N}\big)-F_N(t,u)\Big| \; \d u \d t.
\]
From the Cauchy-Schwarz inequality, we have for any $k \in\T_N$ and $u \in [\frac x N, \frac{x+1} N]$, 
\[ \int_0^T  \Big|F_N\big(t,\tfrac{x}{N}\big)-F_N(t,u)\Big| \; \d t \leqslant \int_0^T \int_{\frac x N}^u \big|\partial_u F_N(t,u)\big|\; \d u \d t\leqslant \sqrt T \big|u-\tfrac{x}{N}\big|^\frac12 \big\|\partial_u F_N\big\|_{L^2([0,T]\times\T)}. \]
\cm{One can check that
\begin{align*}\partial_u F_N = & \frac{(\rho_N)^{m-1}\; \partial_{uuu}\rho_N}{1-\rho_N} \\
&  +  3\;\frac{\partial_{uu}\rho_N\; \partial_u \rho_N}{(1-\rho_N)^2}\big((m-1)(\rho_N)^{m-2} + (2-m) (\rho_N)^{m-1} \big) 
\\ & +  \frac{(\partial_u\rho_N)^3}{(1-\rho_N)^3}\; \mathrm{P}(\rho_N), 
\end{align*}
where 
\[\mathrm{P}(\rho)=  \rho^{m-3}\big(\rho(2-m)+m-1\big)\big(\rho(4-m) + m-2\big) + (2-m)\rho^{m-2}(1-\rho).\]
}
Therefore, from Proposition \ref{lem:Lip} and Proposition \ref{prop:xxx}, one easily obtains  that \cm{there exists $C=C(T,m,C_{\rm Lip},C_h)$ such that}
\[ \big\| \partial_u F_N\big\|_{L^2([0,T]\times \T)} \leqslant \cm{C(\varepsilon_N)^{3-3m}}.\]
Finally we have
\[ \int_0^T\Big|\frac1N\sum_{x\in\T_N} F_N\big(t,\tfrac x N\big) - \int_{\T} F_N(t,u)\; \d u \Big| \d t \leqslant \cm{\frac{C(\varepsilon_N)^{3-3m}}{N^\frac12}},\] which vanishes as $N\to\infty$ from Assumption \ref{ass:eps}. 
\end{proof}

\subsection{Average over large boxes}\label{ssec:average}

To end the proof of Proposition \ref{thm:entropy}, we want to take advantage of the Taylor expansion that seems to arise in \eqref{eq:taylor1} and \eqref{eq:taylor2}. Note that the factor in front of $(\eta(x)-\rho_N(t,\frac x N))$ in that expression can be simplified as: 
\[  \partial_t\lambda_N(t,u)=\partial_{uu}\lambda_N(t,u)\overline{h}'(\rho_N(t,u))+(\partial_u\lambda_N)^2(t,u)\overline g'(\rho_N(t,u)). \] 
First of all, we are going to replace $\eta(x)$ by its empirical average over large boxes. More precisely, let us estimate the error (integrated in time) made by this replacement, which writes as follows
\[ 
\varepsilon_{N,\ell}(T):=\bigg|\int_0^T \int \sum_{x\in\T_N}\partial_t\lambda_N\big(t,\tfrac x N\big) \Big(\eta(x)-\eta^{(\ell)}(x)\Big) \; \d \mu_t^N \; \d t \bigg|,
\] 
 where for any $\ell \in \bb N$, we denote by $\eta^{(\ell)}(x)$ the space average of the configuration $\eta$ on the box of size $2\ell+1$ centered around $x$:
\[\eta^{(\ell)}(x) = \frac{1}{2\ell+1}\sum_{|y-x|\leqslant \ell} \eta(y).\] 
Performing an integration by parts, using the Taylor formula and the Cauchy-Schwarz inequality, one can easily show that for any $\ell \in \bb N$, there exists a constant $C(\ell)>0$ such that 
\[ \varepsilon_{N,\ell}(T) \leqslant  C(\ell) \bigg( \iint_{[0,T]\times\T} \big|\partial_u \partial_t\lambda_N(t,u)\big|^2\; \d t \d u\bigg)^{\frac12} \leqslant \cm{C(\ell)(\varepsilon_N)^{3-3m}},\]
 the last inequality following from Proposition \ref{prop:lambdaN}. 
Therefore, under Assumption \ref{ass:eps}, 
\[ \lim_{\ell \to \infty} \lim_{N\to\infty} \frac{\varepsilon_{N,\ell}(T) }{N}= 0.\]

\medskip

\noindent The next step consists in replacing in \eqref{eq:taylor1} the local function $\tau_xh(\eta)$ by the spatial average 
\[\frac{1}{2\ell+1}\sum_{|y-x|\leqslant \ell} \tau_yh(\eta)\] for $\ell$ sufficiently large and then 
by its mean value $\overline h(\eta^{(\ell)}(x))$.
In the same way, in \eqref{eq:taylor2} we will replace $\tau_xg(\eta)$ by $\overline g(\eta^{(\ell)}(x))$. This step is more involved, and is done thanks to the \emph{one-block estimate} proved in the following section. Once again, because of the degeneracy of the limit profile $\rho(t,\cdot)$ (which can vanish), new arguments are needed w.r.t.\@ \cite{GLT}.

\subsection{The one-block estimate} \label{ssec:one-block}

\begin{lemma}[One-block estimate] \label{lem:one-block}
Let $\varepsilon>0$. For every  local function $\psi:\{0,1\}^{\Z}\to \bb R$ there exists $\gamma_0>0$ and $L_0<\infty$ such that: for all $\ell\geq L_0$ there exists $N_0=N_0(\ell)$ such that for any $N\geq N_0$ we have
\begin{equation}\label{eq:one-block-lem}
\int_0^T\int\frac{1}{N}\sum_{x\in \T_N}\tau_xV_{\ell,\psi}(\eta)f_t^{N}(\eta)\nu_\alpha(\d\eta)\d t\leq \frac{1}{\gamma_0 N}\int_0^T\mc H_N(t)\d t+\varepsilon,
\end{equation}
where \[V_{\ell,\psi}(\eta):=\bigg|\frac{1}{2\ell +1} \sum_{|y|\leqslant \ell} \tau_y\psi(\eta) - \overline{\psi}\big(\eta^{(\ell)}(0)\big)\bigg|.\]
\end{lemma}
We will apply Lemma \ref{lem:one-block} with $\psi(\eta)=h(\eta)$ and $g(\eta)$.

\begin{proof}[Proof of Lemma \ref{lem:one-block}] For $x\in\T_N$, $\ell\in\mathbb{N}$, let \[\mathcal{Q}_{x,\ell}=\Big\{ \eta\ \colon\ \sum_{y=x-\ell}^{x+\ell - 1}\eta(x)\eta(x+1)\geq 1\Big\} \] the set of configurations in which there are two neighbouring particles within distance $\ell$ of $x$ (in particular the box of radius $\ell$ around $x$ contains a \emph{mobile cluster}). We split the left hand side in \eqref{eq:one-block-lem} as follows:
\begin{eqnarray}
\int_0^T\int\frac{1}{N}\sum_{x\in \T_N}\tau_xV_{\ell,\psi}(\eta)\mathbf{1}_{\mathcal{Q}_{x,\ell}}(\eta) f_t^{N}(\eta)\nu_\alpha(\d\eta)\d t\qquad\label{eq:oneblockirreducible}\\
\qquad+\ \int_0^T\int\frac{1}{N}\sum_{x\in \T_N}\tau_xV_{\ell,\psi}(\eta)\mathbf{1}_{\mathcal{Q}^c_{x,\ell}}(\eta) f_t^{N}(\eta)\nu_\alpha(\d\eta)\d t.\label{eq:oneblockreducible}
\end{eqnarray}
As indicated in \cite[Section 3.1]{GLT}, the restriction to the irreducible set $\mc Q_{x,\ell}$ in \eqref{eq:oneblockirreducible} allows us to repeat standard arguments, and to conclude that
\begin{equation}
\limsup_{\ell\rightarrow\infty}\limsup_{N\rightarrow\infty}\int_0^T\int\frac{1}{N}\sum_{x\in \T_N}\tau_xV_{\ell,\psi}(\eta)\mathbf{1}_{\mathcal{Q}_{x,\ell}}(\eta) f_t^{N}(\eta)\nu_\alpha(\d\eta)\d t=0.
\end{equation}
Let us now deal with the other term \eqref{eq:oneblockreducible}. By the entropy inequality \eqref{eq:entrop}, the term inside the time integral $\int_0^T$ can be bounded above by
\begin{equation}
\frac{H\big(\mu_t^N|\nu^N_{\rho_N(t,\cdot)}\big)}{\gamma N}+\frac{1}{\gamma N}\log\int\exp\bigg(\gamma\sum_{x\in\T_N}\tau_x V_{\ell,\psi}(\eta)\mathbf{1}_{\mathcal{Q}_{x,\ell}^c}(\eta)\bigg)\nu^N_{\rho_N(t,\cdot)}(\d\eta)
\end{equation}
for any $\gamma>0$. Recall that $\varepsilon >0$ is fixed. We need to show that we can choose $\gamma>0$ such that 
\begin{equation}\label{eq:oneblockgoal}
\limsup_{\ell\rightarrow\infty}\limsup_{N\rightarrow\infty}\int_0^T\frac{1}{\gamma N}\log\int\exp\bigg(\gamma\sum_{x\in\T_N}\tau_x V_{\ell,\psi}(\eta)\mathbf{1}_{\mathcal{Q}_{x,\ell}^c}(\eta)\bigg)\nu^N_{\rho_N(t,\cdot)}(\d\eta) \d t\leq\varepsilon.
\end{equation} 
Now, contrary to \cite{GLT}, we made no assumption to ensure that $\nu^N_{\rho(t,\cdot)}(\mathcal{Q}_{x,\ell}^c)$ decays exponentially in $\ell$ for all $x$. In fact, this is plain wrong when $\rho(t,\cdot)$ vanishes on an interval. 

\medskip

Let $\ell_0$ be such that the support of $\psi$ is contained in $\{-\ell_0,\dots,\ell_0\}$ and $C:=2\|\psi \|_\infty$ (which clearly does not depend on $\ell$). 
From the uniform convergence stated and proved in Proposition \ref{prop:unif}, we know that there exists a vanishing sequence of positive numbers $(\delta_N^{(1)})$ such that: for any $u\in\T$, any $t\in[0,T]$, and $N\in\mathbb{N}$,
\begin{equation}
\label{eq:eps1} \big|\rho_N(t,u)-\rho(t,u)\big| \leqslant \delta_N^{(1)}.
\end{equation} 
Since $\rho_N\ge \eps_N$ (see Proposition~\ref{prop:max}) while $\rho$ can be equal to $0$, it is natural to impose that  $\delta_N^{(1)} \geq \eps_N$. Without loss of generality, 
we can assume that the sequence $\big(\delta_N^{(1)}\big)$ is decreasing. Moreover, the sequence  $(\rho_N)$ is equicontinuous on $[0,T]\times \T$, and therefore, there exists a nondecreasing continuous modulus of continuity $w:[0,1] \to \R_+$ with $w(0)=0$ such that for any $u,v\in\T$, $t\in[0,T]$, $\varepsilon >0$ and $N\in\mathbb{N}$, 
\begin{equation}
\label{eq:eps2}
|u-v| \leqslant \varepsilon \quad \Rightarrow \quad
\big|\rho_N(t,u)-\rho_N(t,v)\big| \leqslant w(\varepsilon).
\end{equation}
Let us denote 
\begin{equation}\label{eq:eta}\delta_N:=\delta_N^{(1)} + w\big(\tfrac{\ell+\ell_0+1}{N}\big) \xrightarrow[N\to\infty]{}0, \end{equation}
then it follows from the monotonicity of $\big(\delta_N^{(1)}\big)_N$ and of $w$ that $\left(\delta_N\right)_N$ is decreasing.

\medskip

We are going to split $\T_N$ into three sets of points: the \emph{good} ones, the \emph{almost zeroes} and the \emph{bad} ones. Namely, for any $\delta >0$, and any vanishing sequence $(\alpha_N)$ such that $\alpha_N \leqslant \delta$, let
\begin{eqnarray}
G_t^{N,\ell}(\delta)&:=&\Big\{ x\in\T_N\ \colon\ \rho_N(t,\cdot)\geq\delta \quad \text{ on }\big[\tfrac{x-\ell-\ell_0}{N},\tfrac{x+\ell+\ell_0}{N}\big]\Big\},\\
Z_t^{N,\ell}(\alpha_N)&:=&\Big\{ x\in\T_N\ \colon\ \rho_N(t,\cdot)\leqslant \alpha_N \quad \text{ on }\big[\tfrac{x-\ell-\ell_0}{N},\tfrac{x+\ell+\ell_0}{N}\big]\Big\},\\
B_t^{N,\ell}(\delta,\alpha_N)&:=&\T_N\setminus\big(G_t^{N,\ell}(\delta)\cup Z_t^{N,\ell}(\alpha_N)\big).
\end{eqnarray}
The parameters $\delta >0$ and $\alpha_N \to 0$ will be chosen ahead. 
We want to study the limit as $N\to\infty$ of the cardinality of these sets of points (renormalized by $N$).
For that purpose, let us introduce the following sets: for  any $\delta >0$, $t\in[0,T]$, let 
\begin{align*}
  {\mc G}_t(\delta)&:=\Big\{ u\in\T \; : \;  \rho(t,u) \geq \delta\Big\},\\
 {\mc Z}_t(\delta)&:=\overbrace{\Big\{ u\in\T \; : \;  \rho(t,u)=0\Big\}}^{\circ} \cup \Gamma_t(\delta), \\
 {\mc B}_t(\delta)&:=\Big\{ u\in\T \; : \; 0 < \rho(t,u) < \delta\Big\} \subset \Gamma_t(\delta),
\end{align*}
where $\Gamma_t(\delta)$ has been defined in \eqref{eq:gammadelta}. Note first that   \begin{align*}
\T \setminus \big(\mathcal{G}_t(\delta)\cup \mathcal{Z}_t\big) = \mathcal{B}_t(\delta) \cup \Gamma_t,
\end{align*}
where $\mathcal{Z}_t$ and $\Gamma_t$ have been defined respectively in \eqref{eq:Z} and \eqref{eq:gamma}. Therefore, since $\mathrm{Leb}(\Gamma_t)=0$ (recall the proof of Proposition \ref{lem:pos}) the two remaining sets above have the same Lebesgue measure: 
\begin{equation}\label{eq:leb} 
\mathrm{Leb} \big(\mathcal{B}_t(\delta)\big) = \mathrm{Leb}\Big(\T \setminus \big(\mathcal{G}_t(\delta)\cup \mathcal{Z}_t\big)\Big).
\end{equation}
We are going to make use of the following lemma:
\begin{lemma}\label{lem:incl}
Recall that $\delta_N$ has been defined in \eqref{eq:eta}. 
For any $\ell,\ell_0 \in \mathbb{N}$ fixed, and $\delta >0$, the following convergences hold:
\begin{align}
&\lim_{N\to\infty} \frac{1}{N}\Big|G_t^{N,\ell}(\delta-\delta_N)\Big| = \mathrm{Leb}\big( \mathcal{G}_t(\delta)\big)\label{eq:convG} \\
&\lim_{N\to\infty} \frac{1}{N}\Big|Z_t^{N,\ell}(\delta_N)\Big| = \mathrm{Leb}\big( \mathcal{Z}_t\big)\label{eq:convZ}
\end{align}
and therefore from \eqref{eq:leb}: \[ \lim_{N\to\infty} \frac{1}{N}\Big|B_t^{N,\ell}(\delta-\delta_N,\delta_N)\Big| = \mathrm{Leb}\big( \mathcal{B}_t(\delta)\big).\]
\end{lemma}
We will prove  Lemma \ref{lem:incl} further. Let us first end the proof of Lemma \ref{lem:one-block}, more precisely of \eqref{eq:one-block-lem}. Fix $\delta >0$ as a parameter that will vanish at the end of this paragraph, \emph{after} letting $N\to\infty$ and $\ell \to \infty$. Take the expression under the limit in the left hand side of \eqref{eq:oneblockgoal}, and take $N$ sufficiently large such that $\delta-\delta_N > \delta_N$.  We divide the sum that appears there into three sums: \begin{itemize}
\item one over $B_t^{N,\ell}(\delta-\delta_N,\delta_N)$,
\item one over $Z_t^{N,\ell}(\delta_N)$,
\item and the last one over $G_t^{N,\ell}(\delta-\delta_N)$,
\end{itemize} since by definition their union gives $\T_N$.  We bound each sum as follows: first, since $V_{\ell, \psi}(\eta)$ is bounded by $C$, we have
\[\sum_{x\in B_t^{N,\ell}(\delta-\delta_N,\delta_N)} \tau_x V_{\ell,\psi}(\eta) \mathbf{1}_{\mc Q_{x,\ell}^c}(\eta) \leqslant C \Big|B_t^{N,\ell}(\delta-\delta_N,\delta_N)\Big|.\]
Then note that the two sums over $Z_t^{N,\ell}(\delta_N)$ and $G_t^{N,\ell}(\delta-\delta_N)$
are functions with disjoint supports; since the measure is product, the average
factorizes.
To bound the term with the sum over $Z_t^{N,\ell}(\delta_N)$, note that, if $\eta(x+y)=0$ for all $|y|\leq \ell +\ell_0$, then $\tau_x V_{\ell,\psi}(\eta)=0$. Moreover, if a non-decreasing function\footnote{A function $f:\{0,1\}^{\T_N}\rightarrow\R$ is said to be non-decreasing if $f(\eta)\leq f(\eta')$ as soon as $\eta\preccurlyeq\eta'$, where $\preccurlyeq$ denotes the coordinate-wise order in $\{0,1\}^{\T_N}$
} has support in
 $\big\{x\in\T_N\; \colon\; \rho_N(t,\frac x N)\leqslant \delta_N\big\}$, we can replace $\nu^N_{\rho_N(t,\cdot)}$ by the homogeneous measure $\nu_{\delta_N}$ when overestimating its expectation (as we do in the second inequality below). Consequently, we can bound
\begin{align*}
\int\exp\bigg(\gamma\sum_{x\in Z_t^{N,\ell}(\delta_N)}&\tau_x V_{\ell,\psi}(\eta)\mathbf{1}_{\mathcal{Q}_{x,\ell}^c}(\eta)\bigg)\nu^N_{\rho_N(t,\cdot)}(\d\eta)\\
&\leqslant \int\exp\bigg(\gamma C\sum_{x\in Z_t^{N,\ell}(\delta_N)}\mathbf{1}_{\big\{\exists\; |y|\leq \ell+\ell_0\colon\ \eta(x+y)=1\big\}}\bigg)\nu^N_{\rho_N(t,\cdot)}(\d\eta)\\
&\leqslant \int\exp\bigg(\gamma C\sum_{x\in Z_t^{N,\ell}(\delta_N)}\mathbf{1}_{\big\{\exists\; |y|\leq \ell+\ell_0\colon\ \eta(x+y)=1\big\}}\bigg)\nu_{\delta_N}(\d\eta) \\
&\leqslant \int\exp\bigg(\gamma C\sum_{y\in \T_N}\eta(y)(2\ell+2\ell_0+1)\bigg)\nu_{\delta_N}(\d\eta)\\
&\leqslant \Big(\delta_N\big(e^{\gamma C(2\ell+2\ell_0+1)}-1\big)+1\Big)^N.
\end{align*} 
Finally, for any $x \in G_t^{N,\ell}(\delta-\delta_N)$, and any $t\in[0,T]$, we know that 
\[ \nu_{\rho_N(t,\cdot)}^N \big( \mc Q_{x,\ell}^c\big) \leqslant \big(1-(\delta-\delta_N)^2\big)^\ell.  \]
Therefore, we bound the  term under the limit in \eqref{eq:oneblockgoal} as follows:
\begin{align}
\int_0^T&\frac{1}{\gamma N} \log \int\exp\bigg(\gamma\sum_{x\in\T_N}\tau_x V_{\ell,\psi}(\eta)\mathbf{1}_{\mathcal{Q}_{x,\ell}^c}(\eta)\bigg)\nu^N_{\rho_N(t,\cdot)}(\d\eta)\d t\notag \\
& \leqslant  \int_0^T\frac{C \big|B_t^{N,\ell}(\delta-\delta_N,\delta_N)\big|}{N} \d t\notag\\
& \quad + \frac{T}{\gamma}\log\big(1+\big(e^{\gamma C(2\ell+2\ell_0+1)}-1\big)\delta_N\big) \notag\\
& \quad  + \int_0^T\frac{1}{\gamma N (2\ell+1)} \sum_{x\in G_t^{N,\ell}(\delta-\delta_N)} \log \Big(\nu_{\rho_N(t,\cdot)}^N \big( \mc Q_{x,\ell}^c\big)\big(\exp(\gamma(2\ell+1)C)-1\big)+1\Big)\d t\notag \\
& \leqslant \int_0^T \frac{C \big|B_t^{N,\ell}(\delta-\delta_N,\delta_N)\big|}{ N} \d t\label{eq:term1}\\
& \quad + \frac{T}{\gamma}\log\big(1+\big(e^{\gamma C(2\ell+2\ell_0+1)}-1\big)\delta_N\big) \label{eq:term2}\\
&  \quad + \frac{T}{\gamma (2\ell+1)} \Big( \exp(\gamma(2\ell+1)C)-1 \Big) \big(1-(\delta-\delta_N)^2\big)^\ell. \label{eq:term3}
\end{align}
We first take $N\to\infty$, then $\ell \to \infty$ and then $\delta \to 0$. We treat each term separately: from Lemma \ref{lem:incl}, Fatou's lemma, and the fact that $\mathcal{B}_t(\delta)\subset \Gamma_t(\delta)$ we obtain: 
\begin{align*} \lim_{\delta\to 0}\limsup_{\ell\to\infty}\limsup_{N\to\infty} \int_0^T \frac{1}{N}\Big|B_t^{N,\ell}(\delta-\delta_N,\delta_N)\Big| \; \d t& \leq \lim_{\delta\to 0}\int_0^T \mathrm{Leb}\big(\mathcal{B}_t(\delta)\big)\; \d t \\
& \leqslant \lim_{\delta\to 0}\int_0^T \mathrm{Leb}\big(\Gamma_t(\delta)\big)\; \d t=0,\end{align*}
where the last equality follows from Proposition \ref{lem:pos}. 
The second term \eqref{eq:term2} easily vanishes since $\delta_N\to 0$. Finally, for the last term \eqref{eq:term3}, we choose $\gamma >0$ such that $2\gamma C + \log(1-\delta^2) < 0$  and the result follows. 
\end{proof}

We now go back to the proof of Lemma \ref{lem:incl}. 

\begin{proof}[Proof of Lemma \ref{lem:incl}] The proof is based on the following fact:  for any $\delta >0$ and $N$ sufficiently large,
\begin{align}
\label{eq:incl1}
&\mathcal{G}_t(\delta) \subset \frac{1}{N}G_t^{N,\ell}(\delta-\delta_N) \subset \mathcal{G}_t(\delta-2\delta_N),\\
& {\mathcal{Z}_t} \subset \frac{1}{N}Z_t^{N,\ell}(\delta_N) \subset
 \mathcal{Z}_t(2\delta_N). \label{eq:incl2}
\end{align}
Let us prove the first inclusion in \eqref{eq:incl1}, namely: if $u \in \mathcal{G}_t(\delta)$ then, for any $N$ sufficiently large, $\lfloor uN\rfloor \in G_t^{N,\ell}(\delta-\delta_N)$. 
\medskip

Let $u \in \mathcal{G}_t(\delta)$ and $y \in \big[ \frac{\lfloor uN\rfloor - \ell-\ell_0}{N} ,\frac{\lfloor uN\rfloor + \ell+\ell_0}{N}\big]$, which implies  $|y-u| \leqslant \frac{\ell+\ell_0+1}{N}$. 
Using \eqref{eq:eps1} and \eqref{eq:eps2} we get
\[ \rho_N(t,y) \geqslant \rho_N(t,u)-w\big(\tfrac{\ell+\ell_0+1}{N}\big) \geqslant  \rho(t,u)-\delta_N^{(1)}-w\big(\tfrac{\ell+\ell_0 +1}{N}\big) 
\geqslant \delta -\delta_N,\] 
which proves the claim. The same argument works to prove the symmetric inclusion, namely: if $x \in G_t^{N,\ell}(\delta-\delta_N)$ then $\frac x N \in\mathcal{G}_t(\delta-2\delta_N)$. As a result, \eqref{eq:incl1} follows.
The proof of the second series of inclusions \eqref{eq:incl2} is very similar and we let the reader to check it. 

In order to prove \eqref{eq:convG}, it is enough to use \eqref{eq:incl1} and to show that the Lebesgue measures converge, namely:
\[\mathrm{Leb}\big(\mathcal{G}_t(\delta-2\delta_N)\big) \xrightarrow[N\to\infty]{} \mathrm{Leb}\big(\mathcal{G}_t(\delta)\big).\]
This is indeed the case since $(\delta_N)$ is a decreasing sequence and therefore the family of sets $\big(\mathcal{G}_t(\delta-2\delta_N)\big)$ is decreasing for inclusion. 

Therefore, thanks to the continuity of the Lebesgue measure, there holds
\[
\mathrm{Leb}(\mathcal{G}_t(\delta)) = \mathrm{Leb}\big( \bigcap_{N\ge 1}\mathcal{G}_t(\delta-2\delta_N)\big) = \lim_{N \to \infty}\mathrm{Leb}\big(\mathcal{G}_t(\delta-2\delta_N)\big).
\]
A very similar argument can be worked out to prove \eqref{eq:convZ}. For that case, first note that 
\[\mathcal{Z}_t \cup \Gamma_t = \bigcap_{N\ge 1} \mathcal{Z}_t(2\delta_N),\] and then one is able to conclude the proof, since the boundary set $\Gamma_t$ satisfies $\mathrm{Leb}(\Gamma_t)=0$ (from Proposition \ref{lem:pos}).
\end{proof}

\subsection{Conclusion}

Putting together the computation of the entropy production in Lemma \ref{lem:entropy}, and then the replacements done in Section \ref{ssec:average} and Lemma \ref{lem:one-block}, up to now we have proved the following:

\begin{corollary} \label{cor:replacement}
There exists $\gamma_0>0$ and $\ell_0 \in \bb N$, such that, for any $\ell\geqslant \ell_0$, there exists $N_0 = N_0(\ell_0)$ such that, for any $N\geqslant N_0$, 
\begin{align}
\mc H_N(T) - \mc H_N(0) \leqslant &\frac{1}{\gamma_0}\int_0^T \mc H_N(t)\; \d t + \varepsilon_T(N,\ell) \label{eq:sum0}\\
& + \int_0^T\int \sum_{x \in \T_N} \partial_{uu} \lambda_N\big(t,\tfrac{x}{N}\big) \; \overline H\Big(\eta^{(\ell)}(x),\rho\big(t,\tfrac{x}{N}\big) \Big) \d\mu_{t}^N\; \d t \label{eq:sum1}\\
& + \int_0^T\int \sum_{x \in \T_N} (\partial_{u} \lambda_N)^2\big(t,\tfrac{x}{N}\big) \; \overline G\Big(\eta^{(\ell)}(x),\rho\big(t,\tfrac{x}{N}\big) \Big) \d\mu_{t}^N\; \d t,\label{eq:sum2}
\end{align}
where \begin{align*}
\overline H(a,b)&:=\overline h(a)-\overline h(b) - \overline h'(b)(a-b) \\
\overline G(a,b)&:=\overline g(a)-\overline g(b) - \overline g'(b)(a-b) 
\end{align*}
and 
\[ \limsup_{\ell \to \infty}\limsup_{N\to\infty} \frac{\varepsilon_T(N,\ell)}{N} = 0.\]
\end{corollary}

In this last paragraph we show that \eqref{eq:sum1} and \eqref{eq:sum2} are bounded from above by a constant times \eqref{eq:sum0}. We treat only \eqref{eq:sum1}, the same argument works for \eqref{eq:sum2}. Note that  applying the entropy inequality,  we can bound \eqref{eq:sum1} above by
\[
\frac{1}{\gamma} \int_0^T \mc H_N(t) \; \d t + \frac{1}{\gamma} \int_0^T \log \bb E_{\rho_N(t,\cdot)}^N\bigg[\exp\bigg\{\gamma \sum_{x\in\T_N}  \partial_{uu} \lambda_N\big(t,\tfrac{x}{N}\big) \overline H\big(\eta^{(\ell)}(x),\rho_N\big(t,\tfrac{x}{N}\big)\big)\bigg\}\bigg]\; \d t,
\]
for any $\gamma >0$. The first term will be added to \eqref{eq:sum0}. 
A large deviation argument will allow us to chose $\gamma >0$ such that the second term vanishes: 
\begin{lemma}[Large deviation estimate]\label{lem:largedev}
There exists $\gamma>0$ such that, for all $t\in[0, T]$,
\[
\limsup_{\ell \to \infty}\limsup_{N\to\infty} \frac{1}{N} \log \bb E_{\rho_N(t,\cdot)}^N\bigg[\exp\bigg\{\gamma \sum_{x\in\T_N}  \partial_{uu} \lambda_N\big(t,\tfrac{x}{N}\big) \overline H\big(\eta^{(\ell)}(x),\rho_N\big(t,\tfrac{x}{N}\big)\big)\bigg\}\bigg] \leqslant 0.
\]
\end{lemma}

\begin{proof}[Proof of Lemma \ref{lem:largedev}] We follow the lines of \cite[Chapter 6]{kl}, where the rough argument is well exposed and now standard. The main difference here consists in the presence of the approximate solution $\rho_N$ instead of $\rho$. A Riemann-type convergence like in \eqref{eq:rieman-unif} will be enough to conclude. 
\end{proof}

This concludes the proof of Proposition \ref{thm:entropy}. 

\section*{Acknowledgements}
{We thank Patr\'icia Gon\c calves, Claudio Landim, Cristina Toninelli and Kenkichi Tsunoda for helpful discussions. 

O.B.\ and M.S.\ are grateful to the University of Tokyo for its hospitality. O.B.\ acknowledges support from ANR-15-CE40-0020-03 grant LSD. C.C.\ and M.S.\ thank Labex CEMPI (ANR-11-LABX-0007-01), and C.C.\ is grateful to ANR-13-JS01-0007-01 (project GEOPOR) for their support. O.B.~and M.S.~thank INSMI (CNRS) for its support through the PEPS project ``D\'erivation et \'Etude Math\'ematique de l'\'Equation des Milieux Poreux'' (2016). M.S. was supported by JSPS Grant-in-Aid for Young Scientists (B) No. JP25800068. The research leading to the
present results benefited from the financial support of the seventh Framework Program of the European Union (7ePC/2007-2013), grant agreement n\textsuperscript{o}266638. This project has received funding from the European Research Council (ERC) under  the European Union's Horizon 2020 research and innovative programme (grant agreement  n\textsuperscript{o}715734). 
}

\appendix

\section{Connected components of the positivity set} \label{app:profile}

In this section we prove Lemma \ref{lem:connect}, more precisely:

\begin{lemma}[Connected components of the positivity set]\label{lem:connectbis}
Denote by $\mathcal I_t$ the set of the connected components of $\Pp_t$ for $t \geq 0$. 
Then one can build an injective mapping from $\mathcal I_t$ to $\mathcal I_s$ for all $t \geq s \geq 0$.
In particular, the function $t\mapsto \#\mathcal I_t$ is non-increasing. 
\end{lemma}
\begin{proof}
Let $t>0$ and let $(a,b) \in \mathcal I_t$, i.e., $\rho(t,a) = \rho(t,b) = 0$ and $\rho(t,u) >0$ for $u \in (a,b)$. 
Following \cite[Proposition 14.1]{vaz}, the mapping $t \mapsto \Pp_t$ is monotone:
\[
{\Pp_s} \subset {\Pp_t}, \qquad \text{for all } (s,t) \in [0,T]^2 \;\text{ such that}\; s \leq t. 
\]
This ensures that 
\be\label{eq:appB_1}
\rho(s,a) = \rho(s,b) = 0, \qquad \text{for any } s \in [0,t].
\ee
Since $0 \leq \rho \leq 1$ and \cm{the pressure $\Pr=\frac{m}{m-1}\rho^{m-1}$ is Lipschiz continuous (Proposition \ref{prop:regul-ast})}, 
the weak formulation~\eqref{eq:weak} still holds for test functions $\xi$ of the form 
$\xi(\tau,u) = \theta(\tau) \zeta(u)$ with $\theta \in L^1\cap BV(\R_+)$ and compactly supported, and $\zeta \in H^{1}(\bbT)$ thanks to the density of $\mc C^1([0,T])$ in $BV(0,T)$ and 
of $\mc C^1(\bbT)$ in $H^1(\bbT)$ for the respective weak-$\star$ and weak topologies. 
Here, $BV(\R_+)$ denotes the set of real valued functions of bounded variations on 
$\R_+$, i.e., functions $t \mapsto \theta(t)$ such that 
$\p_t \theta$ is a finite Radon measure on $\R_+$.
Fix $s \in [0,t)$ and $\eps \in (0, (b-a)/2)$, then choose $\xi = \theta \zeta_\eps$ with 
\[
\theta(\tau) = {\bf 1}_{(s,t)}(\tau) \qquad \text{and}\qquad 
\zeta_\eps(u) = \max\left(0, \min\left(1,\frac{u-a}{\eps}, \frac{b-u}{\eps}\right)\right)
\]
in the weak formulation~\eqref{eq:weak}. This provides 
\[
\int_\bbT \rho(s,u) \zeta_\eps(u) \d u = 
\int_\bbT \rho(t,u) \zeta_\eps(u) \d u + \int_s^t\frac1\eps \int_{a}^{a+\eps} \p_u \cm{(\rho^m)} \d u \d \tau 
- \int_s^t \frac1\eps \int_{b-\eps}^{b} \p_u\cm{(\rho^m)} \d u \d \tau.
\]
Using~\eqref{eq:appB_1}, one gets that 
\be\label{eq:appB_2}
\int_\bbT \rho(s,u) \zeta_\eps(u) \d u = 
\int_\bbT \rho(t,u) \zeta_\eps(u) \d u + \frac1\eps\int_s^t \left( \cm{\rho^m}(\tau, a+\eps)  + \cm{\rho^m}(\tau, b-\eps)\right) \d\tau.
\ee
It follows from \eqref{eq:appB_1} and from the Lipschitz continuity of $\Pr$ that there exists $C>0$ such that 
\begin{align*}
0 \leq  \rho^m(\tau, a+\eps) 	& = \big(\rho^{m-1}(\tau,a+\varepsilon)\big)^{\frac{m}{m-1}}\\
& = \big(\rho^{m-1}(\tau,a+\varepsilon)-\rho^{m-1}(\tau,a)\big)^{\frac{m}{m-1}}\leq C \eps^\frac{m}{m-1}, 
\end{align*}
and in the same way 
\[ 0 \leq  \rho^m(\tau, b-\eps) \leq C \eps^\frac{m}{m-1}.\]
These estimates together with the convergence of $\zeta_\eps$ in $L^1(\bbT)$ towards ${\bf 1}_{(a,b)}$ allow
to pass to the limit $\eps \to 0$ in~\eqref{eq:appB_2}, leading to 
\[
\int_a^b \rho(s,u)\d u = 
\int_a^b \rho(t,u) \d u >0, \qquad \text{for any } s \in [0,t].
\]
Since $\rho(s,\cdot)$ is continuous  and because of~\eqref{eq:appB_1}, 
this implies that there exists (at least) one interval $(\alpha, \beta) \subset (a,b)$ such that 
$\rho(s,\alpha) = \rho(s,\beta) = 0$ and $\rho(s,u) >0$ on $(\alpha, \beta)$. Such an interval $(\alpha, \beta)$ belongs to 
$\mathcal I_s$, and the mapping from $\mathcal I_t$ to $\mathcal I_s$ sending $(a,b)$ to $(\alpha, \beta)$ is injective. 
\end{proof}

\end{document}

%% file: PME_BlondelCancesSasadaSimon-corrections.bbl
\begin{thebibliography}{}

\bibitem{BGL} C. Bardos, F. Golse, and C. D. Levermore, \emph{Fluid dynamic limits of kinetic equations.
II. Convergence proofs for the Boltzmann equation}, Comm. Pure Appl. Math. {\bf 46}, 667--753  (1993).

\bibitem{bt}
G. Biroli and C. Toninelli, \emph{Jamming percolation and glassy dynamics}, J. Stat. Phys. {\bf 126}, 731--763 (2007).

\bibitem{blmv} T. Bodineau, J. Lebowitz, C. Mouhot, and C. Villani, \emph{Lyapunov functionals for boundary-driven
nonlinear drift-diffusion equations}, Nonlinearity {\bf 27}(9) 2111--2132 (2014).

\bibitem{CG11} C. Canc\`es and T. Gallou\"et, \emph{On the time continuity of entropy solutions}, {J. Evol. Equ.} {\bf 11} 43--55 (2011).  

\bibitem{CMS} C. Canc\`es, H. Mathis and N. Seguin, \emph{Error estimate for time-explicit finite volume approximation of strong solutions to systems of conservation laws}, SIAM Journal on Numerical Analysis, Society for Industrial and Applied Mathematics, {\bf 54}(2), 1263--1287 (2016).

\bibitem{CarTos} J. A. Carrillo and G. Toscani, \emph{Asymptotic L1-decay of solutions of the porous medium equation to self-similarity}, Indiana Univ. Math. J. {\bf  49}(1):113--142 (2000).

\bibitem{CT80}
M.~G. Crandall and L.~Tartar,  \emph{Some relations between nonexpansive and order preserving mappings}, 
\newblock {Proc. Amer. Math. Soc.}, {\bf 78}(3) 385--390 (1980).

\bibitem{Daf}
C.M. Dafermos, \emph{The second law of thermodynamics and stability}, Arch. Rational Mech. Anal. {\bf 70}, 167--179  (1979).

\bibitem{H2GHs}
E. Di Nezza, G. Palatucci,  and E. Valdinoci,
\emph{Hitchhiker's guide to the fractional {S}obolev spaces}, 
{Bull. Sci. math.} {\bf 136} 521--573 (2012).

\bibitem{Dip} R.J. DiPerna, \emph{Uniqueness of solutions to hyperbolic conservation laws}, Indiana Univ. Math. J. {\bf 28}  137--188 (1979).

\bibitem{free1} 
T. Funaki,  \emph{Free boundary problem from stochastic lattice gas model}, Ann. Inst. H. Poincar\'e, Probab. Statist., {\bf 35} (1999), 573--603

\bibitem{GHMN} T. Gallou\"et, R. Herbin, D. Maltese and A. Novotny, \emph{Error estimates for a numerical approximation to the compressible barotropic Navier--Stokes equations}, IMA Journal of Numerical Analysis {\bf 36}(2), 543--592 (2016). 

\bibitem{GL} G. Giacomin and J.L. Lebowitz, \emph{Phase segregation dynamics in particle systems with long range interactions. I. Macroscopic limits}, J Stat Phys  {\bf 87}: 37 (1997).

\bibitem{GSR} F. Golse and L. Saint Raymond, \emph{The Navier--Stokes limit of the Boltzmann equation for bounded collision kernels}, Inventiones mathematicae {\bf 155} 81--161 (2004).

\bibitem{GLT}
P. Gon\c calves, C. Landim, and C. Toninelli, \emph{Hydrodynamic limit for a particle system with degenerate rates}, {Ann. IHP Probab. Stat.} {\bf 45} 887--909 (2009).

\bibitem{GPV}
M. Z. Guo, G. C. Papanicolau and S. R. S. Varadhan, \emph{Nonlinear diffusion limit for a system with nearest neighbor interactions}, Comm. Math.
Phys. {\bf 118}  31--59 (1988).

\bibitem{JR} V. Jovanovic and C. Rohde, \emph{Error estimates for finite volume approximations of classical solutions for nonlinear systems of hyperbolic balance laws}, SIAM J. Numer. Anal. {\bf 43}(6):2423--2449 (2006).

\bibitem{kl} C. Kipnis and C. Landim. Scaling limits of interacting particle systems, Springer-Verlag, Berlin (1999).
%


\bibitem{LSU}
O.~A. Lady{\v{z}}enskaja, V.~A. Solonnikov, and N.~N. Ural{'}ceva.
\newblock {Linear and quasilinear equations of parabolic type}.
\newblock \emph{Translated from the Russian by S. Smith. Translations of Mathematical
  Monographs}, Vol. 23. American Mathematical Society, Providence, R.I. ( 1967).

\bibitem{LT}C. Lattanzio and A. E. Tzavaras, \emph{Relative entropy in diffusive relaxation}, SIAM J. Math. Anal. {\bf 45}(3), 1563--1584 (2013).

\bibitem{MN}
D. Maltese and A. Novotny, \emph{Compressible Navier--Stokes Equations on Thin Domains}, J. Math. Fluid Mech., to appear.

\bibitem{Otto}
F. Otto, \emph{The Geometry of dissipative equations: the porous medium equation}, Comm. Partial Differential Equations, {\bf 26}(1-2):101--174 (2001).

\bibitem{rs} F. Ritort and P. Sollich, \emph{Glassy dynamics of kinetically constrained models}, Adv. in Phys. {\bf 52}  219--342 (2003).

\bibitem{Spitz} F.~Spitzer, \emph{Interaction of Markov processes},
Advances in Mathematics {\bf 5}(2), 246--290 (1970).

\bibitem{Stampacchia65}
G. Stampacchia.
\newblock \emph{Le probl\`eme de {D}irichlet pour les \'equations elliptiques du
              second ordre \`a coefficients discontinus}
\newblock Ann. Inst. Fourier (Grenoble), {\bf 15}(1), 189--258 (1965).

\bibitem{StR} L. Saint-Raymond, \emph{Hydrodynamic limits:
some improvements of the relative entropy method}, Ann. I. H. Poincar\'e - Analyse Non lin\'eaire {\bf 26}  705--744 (2009).

\bibitem{Tartar_Interpolation}
L.~Tartar.
\newblock {An introduction to {S}obolev spaces and interpolation spaces},
  Vol 3 of {\em Lecture Notes of the Unione Matematica Italiana}.
\newblock Springer, Berlin; UMI, Bologna (2007).

\bibitem{free2} 
K. Tsunoda, \emph{Derivation of a free boundary problem from an exclusion
process with speed change}, Markov Processes Relat. Fields, Vol. 21, (2015) 263-273.

\bibitem{vaz}
J.L. Vazquez. The Porous Medium Equation, Mathematical Theory. Oxford Mathematical Monographs, Clarendon Press, Oxford (2007).



\bibitem{yau} H.~T. Yau, \emph{Relative entropy and hydrodynamics of Ginzburg-Landau models}, Lett. Math. Phys. {\bf 22}(1) 63--80 (1991).




\end{thebibliography}
